\documentclass[10pt,reqno]{article}

\usepackage[b5paper,top=1.2in,left=0.9in]{geometry}
\usepackage{verbatim}
\usepackage{amssymb}
\usepackage{amsmath}
\usepackage{graphicx}
\usepackage{appendix}
\usepackage{color}
\usepackage{amsthm}
\usepackage{tikz}
\usepackage{ifthen}
\usepackage{float}
\usepackage{pdflscape}
\usepackage{rotating}
\usepackage{hyperref}
\usepackage{courier}
\usepackage{multirow}
\usepackage{stackengine}
\usetikzlibrary{arrows,positioning,decorations.pathmorphing,  decorations.markings}

\renewcommand{\d}{\mathrm{d}}
\newcommand{\D}{\mathrm{D}}

\newtheorem{Thm}{Theorem}[section]
\newtheorem{Lem}[Thm]{Lemma}
\newtheorem{Prop}[Thm]{Proposition}
\newtheorem{Cor}[Thm]{Corollary}

\newtheorem*{MainThm}{Main Theorem}

\theoremstyle{definition}
\newtheorem{Def}[Thm]{Definition}
\newtheorem{Rem}[Thm]{Remark}
\newtheorem{Nota}[Thm]{Notation}
\newtheorem{Ex}[Thm]{Example}

\newtheoremstyle{named}{}{}{\itshape}{}{\bfseries}{.}{.5em}{#1 #3}
\theoremstyle{named}

\def\R{\mathbb{R}}
\def\Q{\mathbb{Q}}

\def\C{\mathbb{C}}
\def\Z{\mathbb{Z}}

\def\fD{\mathfrak{D}}

\def\fb{\mathfrak{b}}
\def\sl{\mathfrak{sl}}
\def\g{\mathfrak{g}}

\def\cC{\mathcal{C}}

\def\cG{\mathcal{G}}

\def\cI{\mathcal{I}}

\def\cP{\mathcal{P}}
\def\cR{\mathcal{R}}

\def\cU{\mathcal{U}}

\def\cX{\mathcal{X}}

\def\a{\alpha}
\def\b{\beta}

\def\c{\gamma}
\def\D{\Delta}

\def\d{\delta}

\def\k{\kappa}

\def\l{\lambda}
\def\L{\Lambda}

\def\s{\sigma}

\def\w{\omega}

\def\sss{\stackanchor{3}{1} }

\def\bx{\mathbf{x}}

\def\be{\mathbf{e}}
\def\bE{\mathbf{E}}
\def\bf{\mathbf{f}}
\def\bF{\mathbf{F}}

\def\bH{\mathbf{H}}
\def\bi{\mathbf{i}}
\def\bj{\mathbf{j}}
\def\bJ{\mathbf{J}}

\def\bK{\mathbf{K}}

\def\bp{\mathbf{p}}
\def\bQ{\mathbf{Q}}

\def\bT{\mathbf{T}}

\def\bx{\mathbf{x}}
\def\bY{\mathbf{Y}}

\def\Id{\mathrm{Id}}

\def\imply{\Longrightarrow}
\def\inj{\hookrightarrow}
\def\corr{\longleftrightarrow}

\def\to{\longrightarrow}

\def\o+{\oplus}
\def\bo+{\bigoplus}
\def\x{\times}
\def\<{\langle}
\def\>{\rangle}
\def\oo{\infty}

\def\^{\wedge}
\def\+{\dagger}

\def\inv{^{-1}}
\def\half{\frac12}
\def\dis{\displaystyle}
\def\dd[#1,#2]{\frac{d#1}{d#2}}
\def\del[#1,#2]{\frac{\partial #1}{\partial #2}}
\def\:{\;:\;}

\def\tab{\;\;\;\;\;\;}

\newcommand{\til}[1]{\widetilde{#1}}
\newcommand{\what}[1]{\widehat{#1}}
\newcommand{\xto}[1]{\xrightarrow{#1}}

\newcommand{\mat}[1]{\begin{pmatrix}#1\end{pmatrix}}

\newcommand{\bin}[1]{\begin{bmatrix}#1\end{bmatrix}}

\newcommand{\case}[2][ll]{\left\{\begin{array}{#1}#2 \\ \end{array}\right.}
\newcommand{\Eq}[1]{\begin{align}#1\end{align}}
\newcommand{\Eqn}[1]{\begin{align*}#1\end{align*}}
\renewcommand{\vec}[1]{\overrightarrow{#1}}
\renewcommand{\over}[1]{\overline{#1}}
\tikzset{->-/.style={decoration={
  markings,
  mark=at position .5 with {\arrow{latex}}},postaction={decorate}}}
\tikzset{
    >=latex
    }
\tikzset{>=latex}
\tikzstyle{vthick}=[line width=1.8pt]

\newcommand\drawpath[2]{%
  \foreach \too [count=\c from 1] in {#1}
  {
  \ifthenelse{\c=1}
  {\xdef\from{\too}}
  {\path (\from) edge [->, #2] (\too);
    \xdef\from{\too}}
  };
}
\begin{document}
\title{Quantum cluster mutations and \\reduced word graphs}

\author{  Ivan Chi-Ho Ip\footnote{
	  Department of Mathematics, Hong Kong University of Science and Technology\newline
	  Email: ivan.ip@ust.hk
          }
}
\maketitle

\numberwithin{equation}{section}

\begin{abstract} We give an algebraic proof of the independence of Coxeter moves involved in the construction of positive representations of split-real quantum groups, thus completing a gap in the original construction. To do this, we propose a new quantized version of Lusztig's Injectivity Lemma in the language of quantum cluster algebra, the proof of which by Tits' Lemma reduces to calculations involving sequences of Coxeter moves forming rank 3 cycles. We give a new, constructive proof of Tits' Lemma, and provide the required explicit computation of the quantum cluster mutations under these rank 3 cycles using certain cluster algebraic tricks via universally Laurent polynomials.
\end{abstract}
\vspace{3mm}

\vspace{3mm}


\tableofcontents
%==============================================================
\section{Introduction}\label{sec:intro}

\emph{Positive representations} were introduced in \cite{FI} to study the representation theory of \emph{split real quantum groups} $\cU_{q}(\g_\R)$ associated to semisimple Lie algebra $\g$, as well as its \emph{modular double} $\cU_{qq^\vee}(\g_\R)$ introduced in \cite{Fa1, Fa2} in the regime where $|q|=1$. These representations are natural generalizations of a special family of representations of $\cU_q(\sl_2(\R))$ classified in \cite{Sch} and studied in detail by Teschner \emph{et al.} \cite{PT1, PT2} from the physics point of view of quantum Liouville theory, which is characterized by the actions of \emph{positive (essentially) self-adjoint operators} on the Hilbert space $L^2(\R)$. 

Based on quantizing the regular action on smooth functions on the totally positive flag variety $(G/B)_{>0}$, we constructed, for the simply-laced cases in \cite{FI, Ip2} and non-simply-laced cases in \cite{Ip3}, a family of irreducible representations $\cP_\l$ of $\cU_{qq^\vee}(\g_\R)$ with the Chevalley generators acting on certain Hilbert space as positive self-adjoint operators, parametrized by $\l\in P_{\R^+}$ in the positive real-span of the dominant weights. 

The construction of $\cP_\l$ depends on a choice of reduced expression $\bi_0$ of the longest element of the Weyl group $W$ of $G$, and a crucial claim is that the representations corresponding to different reduced expression $\bi_0'$ are unitary equivalent, and these equivalence should be independent of the choice of Coxeter moves from $\bi_0$ to $\bi_0'$, see Section \ref{sec:qLT:pos} for more details. In \cite{Ip2}, we claimed that this is due to a quantized version of \emph{Lusztig's Injectivity Lemma} \cite[Proposition 2.7]{Lu}. 

To explain in more elementary terms, consider the $3\x 3$ matrices, where we have the following relation for $a,b,c>0$:
\Eq{
\mat{1&a&0\\0&1&0\\0&0&1}\mat{1&0&0\\0&1&b\\0&0&1}\mat{1&c&0\\0&1&0\\0&0&1}=\mat{1&a+c&ab\\0&1&b\\0&0&1}\in U_{>0}^+
}
where $U_{>0}^+$ denotes the totally positive upper unipotent matrices. This can be rewritten using the notation of root subspaces as
\Eq{
(a,b,c)\in\R_{>0}^3\mapsto x_1(a)x_2(b)x_1(c)\in U_{>0}^+
}
Lusztig's Injectivity Lemma (see Lemma \ref{inj} below) states that this map $\R_{>0}^3\to U_{>0}^+$ is injective. Furthermore, there exists a \emph{Coxeter move} $(1,2,1)\sim (2,1,2)$:
\Eq{
\mat{1&a&0\\0&1&0\\0&0&1}\mat{1&0&0\\0&1&b\\0&0&1}\mat{1&c&0\\0&1&0\\0&0&1}=\mat{1&0&0\\0&1&a'\\0&0&1}\mat{1&b'&0\\0&1&0\\0&0&1}\mat{1&0&0\\0&1&c'\\0&0&1}
}
or equivalently
\Eq{\label{121=212}
x_1(a)x_2(b)x_1(c)=x_2(a')x_1(b')x_2(c')
}
where 
\Eq{
a'=\frac{bc}{a+c},\tab b'=a+c,\tab c'=\frac{ab}{a+c}
}
is an involution between the variables.

The above index can clearly be generalized from $(1,2,1)$ to any reduced words $\bi_0$ for general semisimple Lie types, and sequence of Coxeter moves relating different longest word $\bi_0\sim \bi_0'$. As a consequence of Lusztig's Injectivity Lemma, if we have a sequence of Coxeter moves $\bi_0\sim \bi_0'\sim \cdots\sim \bi_0$ that returns to itself, which by \eqref{121=212} induces a sequence of equalities
\Eq{\label{aiai}
x_{i_1}(a_1)\cdots x_{i_N}(a_N) = x_{i_1'}(a_1')\cdots x_{i_N'}(a_N') =\cdots = x_{i_1}(a_1'')\cdots x_{i_N}(a_N'')\in U_{>0}^+}
then $a_i=a_i''$ for all $i=1,...,N$.

As explained, in \cite{Ip2} we proposed a quantized version of Lusztig's Injectivity Lemma and cited \cite{BZ}, where the variables $a_1,...,a_N$ are non-commutative quantum variables in a quantum torus algebra $\bT_q$, and stated that if $a_i$, $a_i''$ are as in \eqref{aiai}, then they are the same elements in $\bT_q$. However, in a private communication, Linhui Shen kindly pointed out that this reference is not accurate, since the equation \eqref{121=212} is not well-defined for general Lie types, as we do not really have a well-defined notion of ``quantum matrices" other than type $A_n$. Furthermore, the classical proof using the Bruhat decomposition cannot carry over to the quantum case, as the definition of the \emph{quantum Bruhat cells} (which belongs to the dual space) used by \cite{BZ} does not match with the setup of our proposed quantum Lusztig transformation. Therefore, the quantum analogue of the injectivity Lemma is actually unclear and non-trivial. In fact, this result was also assumed implicitly in the construction of the non-simply-laced case \cite{Ip3} without proof. Therefore this posed a gap in the completeness of the general construction of positive representations, which was criticized by Goncharov--Shen in \cite[Section 5.2.6]{GS}. 

In \cite{Ip7}, based on the idea from \cite{FG2, SS1}, a quantum cluster realization of the positive representations of $\cU_q(\g_\R)$ was proposed, and the quantum Lusztig's Injectivity Lemma can be understood as the independence of choice of quantum cluster mutations associated to different Coxeter moves between reduced word $\bi_0\sim \bi_0'$. Unfortunately, again the statement was not explicitly proved since the cluster realization solely relies on the explicit algebraic expressions of the Chevalley generators obtained from \cite{FI, Ip2, Ip3} directly without changes. In Goncharov--Shen \cite{GS}, the positive representations are instead essentially obtained geometrically by canonically quantizing certain regular functions, known as the potential functions, of cluster varieties arising from the moduli spaces of framed $G$-local systems, which do not depend on the choice of reduced expression of $w_0$, and in particular on the choice of Coxeter moves.

In this paper, we attempt to provide an alternative, rigorous treatment to fill the gap of the original construction of the positive representations, using the cluster algebraic realization studied in \cite{Ip7} which will be reviewed in the Preliminaries Section. In brief, given a reduced word $\bi_0$ of the longest element $w_0\in W$ of the Weyl group, we can associate with it a special quantum torus algebra $\bT_q^{\bi_0}$, and for any Coxeter move relating $\bi_0\sim \bi_0'$, there exists a birational involution called the \emph{quantum cluster mutation} $\mu^q:\bT_q^{\bi_0'}\to \bT_q^{\bi_0}$.

The main result (Theorem \ref{mainThm}) of the paper is the following: 
\begin{MainThm}
If $\cC_1, \cC_2: \bi_0\to \bi_0'$ are two sequences of Coxeter moves, and 
$$\mu_{\cC_j}^q: \bT_q^{\bi_0'}\to \bT_q^{\bi_0},\tab j=1,2$$
are the induced compositions of quantum cluster mutations corresponding to the moves $\cC_j$, then $$\mu_{\cC_q}^q=\mu_{\cC_2}^q.$$
\end{MainThm}
As a consequence of this result, we propose in Corollary \ref{qLT} the quantized analogue of Lusztig's Injectivity Lemma by relating the quantum Lusztig's coordinates $a_i$ with the quantum cluster variables $X_i$ of $\bT_q^{\bi_0}$ via a monomial transformation, which completes the gaps in the original construction of the positive representations of $\cU_{q\til{q}}(\g_\R)$.

To stay chronicle to the construction, we will prove the Main Theorem using methods independent of \cite{GS}. This is based on the Tits' Lemma \cite{Tits}, restated in Theorem \ref{Tits}, which reduces the study of the fundamental groupoid of the \emph{reduced words graph} under Coxeter moves, to cycles arising only from rank 3 cases. This reduces the proof of the Main Theorem to checking images of quantum cluster variables under only those sequences of quantum cluster mutations forming rank 3 cycles, and this follows from an application of a well-known result, due to \cite{FG3} and restated in Theorem \ref{Ymut}, of $\s$-periods of quantum tropical $Y$-seed mutations. 

However, both of these results were proved implicitly. In this paper, we wish to be more transparent, and instead give a new, constructive proof of Tits' Lemma, and furthermore provide an explicitly computation of the image of the quantum cluster variables under the required rank $3$ cycles of quantum cluster mutations using several algebraic tricks. We also find it illuminating to demonstrate some general principles in calculation involving quantum cluster mutations, in particular the merit of using \emph{universally Laurent polynomials}.

While the argument and the idea behind the fix mentioned above may be elementary to experts in cluster algebra, unfortunately the setup of the notations and preliminaries required to state our results are quite overwhelming. In order for the paper to be self-contained, we decide to spend a substantial portion of the text to provide explanation of the notations and preliminaries required to state and proof the Main Theorem. This also clarifies, in Section \ref{sec:qLT}, certain relationship between the original construction in \cite{Ip2, Ip3} with the newer language of quantum cluster algebra not realized back then. Particularly important is that we explain the \emph{ad hoc} choice of orientations of Dynkin diagram used in the original construction \cite[Definition 5.3]{Ip2} between the quantum Lusztig's coordinates is nothing but an unimportant choice of arrows between the frozen variables, and therefore unnecessary.

The outline of the paper is as follows. In Section \ref{sec:pre}, we fix notations and provide definitions and main properties of Lusztig's coordinates on totally positive semigroups and their transformations via Coxeter moves, the quantum group $\cU_q(\g)$ and the Drinfeld's double $\fD_q(\g)$, and finally the quantum cluster algebra with the associated quivers required in this paper. In Section \ref{sec:qLT}, we propose a new definition of quantum Lusztig's transformations, and relates our new definition to the one originally proposed in the construction of positive representations \cite{Ip2}, which reduces the problem to proving the Main Theorem. In Section \ref{sec:Tits}, we introduce the terminologies required to state the Tits' Lemma, and in Section \ref{sec:TitsPf},  we give a new constructive proof by induction on the size of the reduced word graphs. Finally in Section \ref{sec:mainPf}, we complete the proof of the Main Theorem by explicitly calculating the quantum cluster mutations forming rank 3 cycles.

\textbf{Acknowledgment.} We thank Linhui Shen for pointing out the logical gap in the original construction and subsequent discussions related to the issue. We also thank Ronald Chun Wai Wong for aiding with the computations in the rank 3 cluster mutations. The author is supported by the Hong Kong RGC General Research Funds ECS \#26303319.

%==================================================================================
\section{Notations and Preliminaries}\label{sec:pre}
In this section, we recall some terminologies required in this paper, keeping the exposition minimal to what is needed. For example, we do not recall the full construction of the basic quiver, nor the Hopf algebra structures of the quantum groups. For more elaborated details see \cite{Ip2, Ip7}, where the choice of most notations in this paper are based on.
%==================================================================================
\subsection{Coxeter Moves and Lusztig's Data}\label{sec:pre:cox}
Let $\g$ be a simple\footnote{The results of this paper naturally extends to the semisimple case by taking direct product.} Lie algebra over $\C$. Let $G$ be the real simple Lie group corresponding to the split real form $\g_\R$ of the Lie algebra $\g$, and let $B, B^-$ be two opposite Borel subgroups containing a split real maximal torus $T=B\cap B^-$. Let $U^+\subset B, U^-\subset B^-$ be the corresponding unipotent subgroups.

Let $I$ be the root index of the Dynkin diagram of $\g$ such that
\Eq{
|I|=n=\mathrm{rank}(\g).}
Let $\Phi$ be the set of roots of $\g$, $\Phi_+\subset \Phi$ be the positive roots, and $\D_+=\{\a_i\}_{i\in I}\subset \Phi_+$ be the positive simple roots. Let $W=\<s_i\>_{i\in I}$ be the Weyl group of $\Phi$ generated by the simple reflections $s_i:=s_{\a_i}$.
\begin{Def} Let $(-,-)$ be a $W$-invariant inner product on the root lattice $\Z\Phi$. We define
\Eq{
a_{ij}:=\frac{2(\a_i,\a_j)}{(\a_i,\a_i)},\tab i,j\in I
}
such that $A:=(a_{ij})$ is the \emph{Cartan matrix}. 

We normalize $(-,-)$ as follows: we choose the symmetrization factors (also called the \emph{multipliers}) $d_i\in\Q$ such that for any $i\in I$,
\Eq{\label{di}d_i:=\frac{1}{2}(\a_i,\a_i)=\case{1&\mbox{$i$ is long root or in the simply-laced case,}\\\frac{1}{2}&\mbox{$i$ is short root in type $B,C,F$,}\\\frac{1}{3}&\mbox{$i$ is short root in type $G$,}}}
and $(\a_i,\a_j)=-1$ when $i,j \in I$ are adjacent in the Dynkin diagram, such that
\Eq{
d_ia_{ij}=d_ja_{ji}.
}
\end{Def}
\begin{Def}
Let $w\in W$. We call a sequence\footnote{We may sometimes omit the comma in $\bi$ for clarity.}
\Eq{\bi=(i_1,...,i_m),\tab i_k\in I} a \emph{reduced word} of $w$ if 
$$w=s_{i_1}s_{i_2}\cdots s_{i_m}$$
is a reduced expression. The set of reduced words of $w$ is denoted by $\cI(w)$. 

We let
\Eq{
l(w):=|\bi|:=m
} to be the length function. 

Let $w_0$ be the unique longest element of $W$. Throughout the paper we let 
\Eq{
N:=l(w_0)}
and denote by $\bi_0$ a reduced word of $w_0$.
\end{Def}

Let $J\subset I$ and denote by $W_J\subset W$ the subgroup generated by $\<s_i\>_{i\in J}$.\footnote{By convention $W_I=W$ and $W_{\emptyset}=\{e\}$.}

\begin{Def} Let $i,j\in I$. We denote by $\bp_{ij}$ the reduced word ending with $j$ of the longest element $w_{ij}\in W_{\{i,j\}}$. i.e. 
$$\bp_{ij}=ij,\;\; jij, \;\;ijij, \mbox{ or }ijijij$$
respectively for $a_{ij}a_{ji}=0,1,2$ or $3$. 

A \emph{Coxeter move} $C_{ij}$ of $\bi$ transforms involutively a reduced word $\bi\in\cI(w)$ to another reduced word $\bi'\in\cI(w)$ by the rule
$$C_{ij}:(...,\bp_{ij},...)\corr (...,\bp_{ji},...).$$ 
at certain letter position, keeping the other letters fixed. More explicitly, we have
\Eq{
(...,i,j,....)&\corr (..., j,i,....), & a_{ij}a_{ji}=0\label{Cox0}\\
(..., j,i,j,....)&\corr (..., i,j,i,....), & a_{ij}a_{ji}=1\label{Cox1}\\
(...,i,j,i,j,....)&\corr (..., j,i,j,i,....), & a_{ij}a_{ji}=2\label{Cox2}\\
(i,j,i,j,i,j)&\corr (j,i,j,i,j,i), & a_{ij}a_{ji}=3\label{Cox3}
}
The move \eqref{Cox0} will be referred to as the \emph{commutative Coxeter move}. In this paper we will not need to consider the move \eqref{Cox3}.
\end{Def}

Next we recall Lusztig's definition of the totally positive semigroup. Let $\bi_0=(i_1,...,i_N)$ be a reduced word of $w_0$.
\begin{Def}\label{U+} For any $i\in I$, there exists a homomorphism $SL_2(\R)\to G$ denoted by
\begin{eqnarray}
\mat{1&a\\0&1}&\mapsto& x_i(a)\in U_i^+,\\
\mat{b&0\\0&b\inv}&\mapsto &\chi_i(b)\in T,\\
\mat{1&0\\c&1}&\mapsto &y_i(c)\in U_i^-,
\end{eqnarray}
called the \emph{pinning} of $G$, where $U_i^+$ and $U_i^-$ are the simple root subgroups of the unipotent subgroup $U^+$ and $U^-$ respectively. The \emph{positive unipotent semigroup} $U_{>0}^+$ is defined as the image of the map $\iota:\R_{>0}^N\to U^+$ given by
\Eq{\label{w0coord}\iota:(a_1,a_2,...,a_N)\mapsto x_{i_1}(a_1)x_{i_2}(a_2)...x_{i_N}(a_N).}
\end{Def}
\begin{Lem}\cite[Proposition 2.7]{Lu}\label{inj} The map $\iota:\R_{>0}^N\to U^+$ is injective: if 
\Eq{x_{i_1}(a_1)x_{i_2}(a_2)...x_{i_N}(a_N)=x_{i_1}(a'_1)x_{i_2}(a'_2)...x_{i_N}(a'_N),}
then $a_k=a'_k$ for every $k=1,...,N$.
\end{Lem}
\begin{proof} By moving a root subgroup to the other side we obtain
$$x_{i_1}(a_1-a_1')x_{i_2}(a_2)...x_{i_N}(a_N)=x_{i_2}(a'_2)...x_{i_N}(a'_N),$$
hence both sides belong to different Bruhat cells unless $a_1-a_1'=0$. The claim then follows by induction.
\end{proof}
\begin{Def} We define the totally positive semigroup to be
\Eq{\label{gauss}
G_{>0}:=U_{>0}^- T_{>0} U_{>0}^+,
} where $U_{>0}^\pm$ is as above, and $T_{>0}$ are generated by the images $\chi_i(b)$ with $b\in\R_{>0}$.
\end{Def}
In this paper, we are only interested in the identities related to the Coxeter moves \eqref{Cox1}--\eqref{Cox2}, namely
\begin{Prop} We have the following identities in $U_{>0}^+$: for $a,b,c\in \R_{>0}$ and $i,j\in I$,
\begin{itemize}
\item In the simply-laced case (i.e. $a_{ij}a_{ji}=1$), \eqref{Cox1} corresponds to the identity
\Eq{
\label{121}x_i(a)x_j(b)x_i(c)&=x_j(a')x_i(b')x_j(c')
}
where
\Eq{
a'=\frac{bc}{a+c},\tab b'=a+c,\tab c'=\frac{ab}{a+c}.
}
\item In the doubly-laced case (i.e. $a_{ij}a_{ji}=2$), \eqref{Cox2} corresponds to the identity
\Eq{
\label{1212}x_i(a)x_j(b)x_i(c)x_j(d)&=x_j(b')x_i(a')x_j(d')x_i(c'),
}
where
\Eq{
a'=\frac{S}{R},&& b'=\frac{bc^2d}{S},&&c'=\frac{abc}{R},&& d'=\frac{R^2}{S},
}
and
\Eq{R=ab+ad+cd,\tab S=a^2b+d(a+c)^2.}
\end{itemize}
In both cases, the transformation $a,b,c (,d)\mapsto a',b',c'(,d')$ is an involution.
\end{Prop}
In particular since the transformation of the variables is an involution, we conclude that the image of $\R_{>0}^N$ in \eqref{w0coord} is independent of the choice of $\bi_0\in\cI(w_0)$.

%==================================================================================
\subsection{Quantum Groups $\cU_q(\g)$ and $\fD_q(\g)$}\label{sec:pre:nota}
For any finite dimensional complex semisimple Lie algebra $\g$, Drinfeld \cite{Dr} and Jimbo \cite{Ji1} associated to it a remarkable Hopf algebra $\cU_q(\g)$ known as \emph{quantum group}, which is certain deformation of the universal enveloping algebra. We follow the notations used in \cite{Ip7} for $\cU_q(\g)$ as well as the Drinfeld's double $\fD_q(\g)$ of its Borel part.

In the following, we assume again that $\g$ is of simple Dynkin type.
\begin{Def} \label{qi} Let $d_i$ be the multipliers \eqref{di}. We define
\Eq{
q_i:=q^{d_i},
} which we will also write as
\Eq{
q_l&:=q,\\ 
q_s&:=\case{q^{\frac12}&\mbox{if $\g$ is of type $B, C, F$},\\q^{\frac13}&\mbox{if $\g$ is of type $G$},}
}
for the parameters corresponding to long and short roots respectively.
\end{Def}

\begin{Def}\label{Dgdef} We define $\fD_q(\g)$ to be the $\C(q_s)$-algebra generated by the elements
$$\{\bE_i, \bF_i,\bK_i^{\pm1}, \bK_i'^{\pm1}\}_{i\in I}$$ 
subject to the following relations (including the obvious relations involving $\bK_i^{-1}$ and ${\bK_i'}^{-1}$ which we omit below for simplicity):
\Eq{
\bK_i\bE_j&=q_i^{a_{ij}}\bE_j\bK_i, &\bK_i\bF_j&=q_i^{-a_{ij}}\bF_j\bK_i,\label{KK1}\\
\bK_i'\bE_j&=q_i^{-a_{ij}}\bE_j\bK_i', &\bK_i'\bF_j&=q_i^{a_{ij}}\bF_j\bK_i',\label{KK2}\\
\bK_i\bK_j&=\bK_j\bK_i, &\bK_i'\bK_j'&=\bK_j'\bK_i', &\bK_i\bK_j' = \bK_j'\bK_i,\label{KK3}\\
&&[\bE_i,\bF_j]&= \d_{ij}\frac{\bK_i-\bK_i'}{q_i-q_i\inv},\label{EFFE}
}
together with the \emph{Serre relations} for $i\neq j$:
\Eq{
\sum_{k=0}^{1-a_{ij}}(-1)^k\bin{1-a_{ij}\\k}_q\bE_i^{k}\bE_j\bE_i^{1-a_{ij}-k}&=0,\label{SerreE}\\
\sum_{k=0}^{1-a_{ij}}(-1)^k\bin{1-a_{ij}\\k}_q\bF_i^{k}\bF_j\bF_i^{1-a_{ij}-k}&=0,\label{SerreF}
}
where the \emph{$q$-binomial} is given by
\Eq{
\bin{n\\k}_q:=\frac{[n]_q!}{[k]_q![n-k]_q!}\label{qbin}
} with $\dis [k]_q:=\frac{q^k-q^{-k}}{q-q\inv}$ the \emph{$q$-number} and $\dis [n]_q!:=\prod_{k=1}^n [k]_q$ the \emph{$q$-factorial}. 
\end{Def}

\begin{Def}
The quantum group $\cU_q(\g)$ is defined as the quotient
\Eq{
\cU_g(\g):=\fD_q(\g)/\<\bK_i\bK_i'=1\>_{i\in I},
}
and it inherits a well-defined Hopf algebra structure from $\fD_q(\g)$. 
\end{Def}
\begin{Rem} $\fD_q(\g)$ is the \emph{Drinfeld's double} of the quantum Borel subalgebra $\cU_q(\fb)$ generated by $\bE_i$ and $\bK_i$.
\end{Rem}
\begin{Def}We define the rescaled generators
\Eq{
\be_i &:=\left(\frac{\sqrt{-1}}{q_i-q_i\inv}\right)\inv \bE_i,&\bf_i &:=\left(\frac{\sqrt{-1}}{q_i-q_i\inv}\right)\inv \bF_i.\label{rescaleFF}
}
Then we can also consider $\fD_q(\g)$ as the $\C(q_s)$-algebra generated by 
\Eq{\{\be_i, \bf_i, \bK_i, \bK_i'\}_{i\in I}
} and $\cU_q(\g)$ as the corresponding quotient, where the generators satisfy all the defining relations \eqref{KK1}--\eqref{SerreF} above, except \eqref{EFFE} which is modified to be 
\Eq{\label{EFFE2}
[\be_i, \bf_j]=\d_{ij} (q_i-q_i\inv)(\bK_i'-\bK_i).
}
\end{Def}
%==================================================================================
\subsection{Quantum Torus Algebra}\label{sec:pre:qta}
In this section we recall some definitions and properties concerning quantum torus algebra and their cluster realizations. One of the main purposes of this section is to fix the normalization used when defining the quantum cluster mutations, which has a vast variety of conventions in the literature \cite{BZ, FG2, GS, KN}.
\begin{Def} A \emph{cluster seed} is a datum 
\Eq{
\bQ=(Q, Q_0, B, D)
} where $Q$ is a finite set, $Q_0\subset Q$ is a subset called the \emph{frozen subset}, $B=(b_{ij})_{i,j\in Q}$ is a skew-symmetrizable $\frac12\Z$-valued matrix called the \emph{exchange matrix}, and $D=\mathrm{diag}(d_j)_{j\in Q}$ is a diagonal $\Q$-matrix called the \emph{multiplier} with $d_j\inv\in \Z$, such that
\Eq{W:=DB=-B^TD} is skew-symmetric. We assume that $b_{ij}\in\Z$ unless both $i,j\in Q_0$. 

The \emph{rank} of $\bQ$ is defined to be the rank of the exchange matrix $B$, or equivalently the rank of the skew-symmetric matrix $W$.
\end{Def}
In the following, we will only consider the case when there exists a \emph{decoration} 
\Eq{\eta:Q\to I
} to the set of simple roots of a simple Dynkin diagram, such that $D=\mathrm{diag}(d_{\eta(j)})_{j\in Q}$ where $(d_i)_{i\in I}$ are the multipliers given in \eqref{di}.

Let $\L_\bQ$ be a $\Z$-lattice with basis $\{\vec{e_i}\}_{i\in Q}$, and let $d=\min_{j\in Q}(d_j)$ \footnote{For general skew-symmetrizable $B$ we define $d=\left(l.c.m.(d_j\inv)_{j\in Q}\right)\inv$.} Also let
\Eq{\label{wij}w_{ij}=d_ib_{ij}=-w_{ij}.}
We define a skew symmetric $d\Z$-valued form $(-,-)$ on $\L_\bQ$ by 
\Eq{\label{biform}
(\vec{e_i}, \vec{e_j}):=w_{ij}.
}

\begin{Def} Let $q$ be a formal parameter. We define the \emph{quantum torus algebra} $\cX_q^{\bQ}$ associated to a cluster seed $\bQ$ to be an associative algebra over $\C[q^d]$ generated by $\{X_i^{\pm 1}\}_{i\in Q}$ subject to the relations
\Eq{\label{XiXj}
X_iX_j=q^{-2w_{ij}}X_jX_i,\tab i,j\in Q.
}
The generators $X_i\in \cX_q^{\bQ}$ are called the \emph{quantum cluster variables}, and they are \emph{frozen} if $i\in Q_0$. We denote by $\bT_q^\bQ$ the non-commutative field of fractions of $\cX_q^\bQ$, which is well defined since $\cX_q^{\bQ}$ is an Noetherian domain \cite{GW}.

Alternatively, $\cX_q^\bQ$ is generated by $\{X_{\l}\}_{\l\in \L_\bQ}$ with $X_0:=1$ subject to the relations
\Eq{q^{(\l,\mu)}X_{\l}X_{\mu} = X_{\l+\mu},\tab \mu,\l\in\L_\bQ.}
\end{Def}
The merit of this notation is that if we define a star structure on the quantum cluster variables by
\Eq{
X_i^*&:=X_i,\tab i\in Q \label{star}\\
q^*&:=q\inv
}
then for any $\l\in\L_\Q$ the generator $X_\l$ is self-adjoint under the anti-involution $*$.
\begin{Nota}\label{xik}
Comparing both realizations of $\cX_q^{\bQ}$, we shall identify 
\Eq{
X_i=X_{\vec{e_i}},
} and define the notation
\Eq{X_{i_1,...,i_k}:=X_{\vec{e_{i_1}}+...+\vec{e_{i_k}}},}
or more generally for $n_1,...,n_k\in \Z$,
\Eq{X_{i_1^{n_1},...,i_k^{n_k}}:=X_{n_1\vec{e_{i_1}}+...+n_k\vec{e_{i_k}}}.}
\end{Nota}

\begin{Def}\label{quiver} 
We associate to each cluster seed $\bQ=(Q,Q_0,B,D)$ with decoration $\eta$ a quiver, denoted again by $\bQ$, with vertices labeled by $Q$ and adjacency matrix $C=(c_{ij})_{i,j\in Q}$, where
\Eq{
c_{ij}=\case{b_{ij}&\mbox{if $d_i=d_j,$}\\w_{ij}&\mbox{otherwise.}}
}
We call $i\in Q$ a \emph{short} (resp. \emph{long}) \emph{node} if $q_i:=q^{d_i}=q_s$ (resp. $q_i=q_l=q$). An arrow $i\to j$ represents the algebraic relation in $\cX_q^\bQ$
\Eq{\label{XiXj2}
X_iX_j=q_*^{-2}X_jX_i} where $q_*=q_s$ if both $i,j$ are short nodes, or $q_*=q$ otherwise.
\end{Def}
Obviously one can recover the cluster seed from the quiver and the multipliers.
\begin{Nota}\label{arrows}
We will use squares to denote frozen nodes $i\in Q_0$ and circles otherwise. We will also use dashed arrows if $c_{ij}=\half$, which only occur between frozen nodes. 

We will represent the algebraic relations \eqref{XiXj2} by thick or thin arrows (see Figure \ref{thick}) for display conveniences (thickness is \emph{not} part of the data of the quiver). Thin arrows only occur in the non-simply-laced case between two short nodes.
\end{Nota}
   
\begin{figure}[H]
\centering
  \begin{tikzpicture}[every node/.style={inner sep=0, minimum size=0.5cm, thick}, x=1cm, y=0.8cm]
\node[draw,circle] (i1) at(0,4) {$i$};
  \node[draw,circle] (j1) at (3,4) {$j$};
   \draw[vthick, ->](i1) to (j1);
\node at (6,4) {$X_iX_j = q^{-2} X_j X_i$}; 
\node[draw] (i2) at(0,3) {$i$};
  \node[draw] (j2) at (3,3) {$j$};
   \draw[vthick, dashed, ->](i2) to (j2);
\node at (6,3) {$X_iX_j = q^{-1} X_j X_i$};
\node[draw,circle] (i3) at(0,2) {$i$};
  \node[draw,circle] (j3) at (3,2) {$j$};
   \draw[thin, ->](i3) to (j3);
\node at (6,2) {$X_iX_j = q_s^{-2} X_j X_i$};
\node[draw] (i4) at(0,1) {$i$};
  \node[draw] (j4) at (3,1) {$j$};
   \draw[thin,dashed, ->](i4) to (j4);
\node at (6,1) {$X_iX_j = q_s^{-1} X_j X_i$};
  \end{tikzpicture}
  \caption{Arrows between nodes and their algebraic meaning.}\label{thick}
  \end{figure}
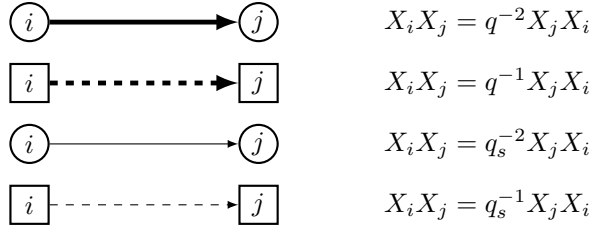

Next we recall two operations on cluster seeds and their induced actions on the quantum torus algebra. Let $\bQ=(Q,Q_0,B, D)$ and $\bQ'=(Q',Q_0',B', D')$ denote two cluster seeds with the same vertex sets $Q=Q'$ and $Q_0=Q_0'$. 
\begin{Def} \label{perm} A \emph{permutation} of a cluster seed $\s:\bQ\to \bQ'$ is a bijection $\s:Q\to Q'$ such that
\Eq{
\s(i)&=i,\tab\forall i\in Q_0,\nonumber\\
b_{ij}'&=b_{\s(i)\s(j)},\\
d_i'&=d_{\s(i)}.\nonumber
}
It induces an isomorphism $\s^*: \bT_q^{\bQ'}\to \bT_q^{\bQ}$ by
\Eq{
\s^*(X_{\s(i)}') := X_i
}
where we denote by $X_i'$ the quantum cluster variables of $\bT_q^{\bQ'}$.
\end{Def}
\begin{Def}\label{qmut} A \emph{cluster mutation in direction $k$}, $\mu_k:\bQ\to \bQ'$ is a bijection $\mu_k:Q\to Q'$ such that $Q_0'=Q_0$ and
\Eq{
b'_{ij} &= \case{-b_{ij}&\mbox{if $i=k$ or $j=k$},\\ b_{ij}+\frac{b_{ik}|b_{kj}|+|b_{ik}|b_{kj}}{2}&\mbox{otherwise},}\\
d'_{i}&=d_i.\nonumber
}
\end{Def}
\begin{Def}
A cluster mutation $\mu_k:\bQ\to\bQ'$ induces an isomorphism $\mu_k^q:\bT_q^{\bQ'}\to \bT_q^{\bQ}$ called the \emph{quantum cluster mutation}, defined by
\Eq{\label{qcm}
\mu_k^q(X_i')=\case{X_k\inv&\mbox{if $i=k$},\\ \dis X_i\prod_{r=1}^{|b_{ki}|}(1+q_k^{2r-1}X_k)&\mbox{if $i\neq k$ and $b_{ki}\leq 0$},\\\dis X_i\prod_{r=1}^{b_{ki}}(1+q_k^{2r-1}X_k\inv)\inv&\mbox{if $i\neq k$ and $b_{ki}\geq 0$},}
}
where again we denote by $X_i'$ the quantum cluster variables of $\bT_q^{\bQ'}$.
\end{Def}
\begin{Rem} It is easy to see that $\mu_k^q$ preserves the $*$-structure \eqref{star}. For example the second expression of \eqref{qcm} can be expanded and rewritten as
\Eq{
\mu_k^q(X_i')=X_i+\sum_{m=1}^n\bin{n\\m}_{q_k}X_{i,k^m},\tab n=|b_{ki}|
}
using the notation \ref{xik}, which is manifestly self-adjoint by \eqref{star} and the fact that the $q$-binomial \eqref{qbin} is invariant under $q\corr q\inv$. Such observation allows for easy consistency check of any formula involving quantum cluster mutations.
\end{Rem}

\begin{Prop}\cite{FG2}
The quantum cluster mutation $\mu_k^q$ can be written as a composition of two homomorphisms
\Eq{
\mu_k^q=\mu_k^\#\circ \mu_k^m,
}
where $\mu_k':\bT_q^{\bQ'}\to \bT_q^\bQ$ is a monomial transformation defined by
\Eq{
\mu_k^m(X_i'):=\case{X_k\inv&\mbox{if $i=k$},\\ X_i&\mbox{if $i\neq k$ and $b_{ki}\leq 0$},\\ X_{i,k^{b_{ki}}}&\mbox{if $i\neq k$ and $b_{ki}\geq 0$},\label{montran}}
}
and $\mu_k^\#:\bT_q^\bQ\to\bT_q^\bQ$ is a conjugation by the \emph{quantum dilogarithm function}
\Eq{
\mu_k^\#:=Ad_{\Psi^{q_k}(X_k)}, 
}where
\Eq{
\Psi^{q}:=\prod_{r=0}^\oo (1+q^{2r+1}x)\inv.
}
\end{Prop}
\begin{Def} A \emph{coframed quiver} $\what{\bQ}$ of $\bQ$ consists of $\bQ$ together with additional frozen vertex $k'$ for each mutable vertex $k$ of $\bQ$ where the multiplier $d_{k'}:=d_k$, and with additional arrow $k\to k'$.
\end{Def}
We will need the following simple observation about mutations of monomials.
\begin{Prop}\label{propmon} Let $\mu_{i_s}\circ\cdots \circ\mu_{i_1}: \bQ'\to \bQ$ be a sequence of cluster mutations. Assume $X_\l\in \bT_q^\bQ$ is a monomial that commutes with any mutable variables $X_k$ in $\bT_q^\bQ$. Then $\mu_{i_1}^q\circ\cdots \mu_{i_s}^q(X_\l)$ remains a monomial in $\bT_q^{\bQ'}$. 

Furthermore, if $\mu_{i_s}\circ\cdots \circ\mu_{i_1}=\s$ is a permutation of the coframed quiver $\what{\bQ}$, then 
\Eq{
\mu_{i_1}^q\circ\cdots \mu_{i_s}^q(X_\l)=\s^*(X_\l).}
\end{Prop}
Here by Definition \ref{perm} if $\l=\sum \l_i\vec{e_i}\in \L$, then
\Eq{
\s^*(X_\l) := X'_{\sum\l_i \vec{e}_{\s\inv(i)}}.
}
\begin{proof}If $X_\l$ commutes with any mutable variables $X_k$, then the conjugation $\mu_k^\#$ is trivial, hence the only effect of a cluster mutation is given by the monomial transformation $\mu_k^m$. Since $\mu_k^q$ is an algebra homomorphism, the image of the cluster mutation remains commutative with the mutable variables in the new seed.

Now as the image of $X_\l$ remains a monomial, its mutation is equivalent to the mutation of its tropicalization. Notice that the monomial transformation $\mu_k^m$ coincides with the tropical $Y$-mutation \cite{KN}, hence the tropicalization of the $X$-variables are just the $c$-vectors, with respect to the coframing in our setting. Hence if the sequence of cluster mutations is a permutation $\s$ of mutable index keeping the coframing fixed, it induces a permutation of the $c$-vectors under the tropicalization, and the claim follows. 

\begin{comment}
 $\mu_{i_m}^m\circ\cdots \mu_{i_1}^m(X_\l)=X_{\l'}'$ is again a self-adjoint element, therefore it suffices to look at the tropicalization of the variables, i.e. the effect of mutation on the lattice vector $\l\in\L_\bQ$. But tropicalization of the $X$ variables is just the $c$-vectors. More precisely, in our convention, if we add a coframing to the initial seed $\bQ$, i.e. for each $i\in Q$ we add an extra vertex $\what{i}$ with an extra arrow $i\to \what{i}$, and denote the new adjacency matrix by $\what{B}$, then the $c$-vectors $\vec{c}_j$ are the \emph{row vectors}
$$(\vec{c}_j)_i=\what{b}_{j,\what{i}},\tab i,j\in Q.$$
Then for each seed, if $X_\l$ with $\l=\sum \l_i \vec{c_i}$ commute with any mutable variables, then $\mu_k(X_\l)=X'_{\sum \l_i \vec{c_i}'}$. Hence if the mutation sequence is a permutation $\s$, the $c$-vectors just get permuted by $\s$ and the claim follows. 
\end{comment}
\end{proof}
%==================================================================================
\subsection{Basic Quivers}\label{sec:pre:basicq}
In this subsection, we recall the construction of the basic quiver $\bQ(\bi)$ required for the quantum group embedding \cite{GS,Ip7}, which helps clarifying the cluster nature of the quantum Lusztig's transformation studied in \cite{Ip2}, to be explained in the next section. However, we do not require the full construction (which comes with extra vertices), but only the reduced part corresponding to a given reduced word.

\begin{Def} Let $i\in I$. The \emph{elementary quiver} $\bJ(i)$ consists of
\begin{itemize}
\item The vertex set
\Eq{Q=Q_0=(I\setminus\{i\})\cup \{i_l\}\cup \{i_r\}} 
which are all frozen;
\item The multipliers $D=(d_j)_{j\in Q}$ which is the pull-back of the multipliers \eqref{di} from $I$ under the natural projection $Q\to I$ sending $\{i_l,i_r\}$ to $i$; and 
\item The adjacency matrix $C=(c_{ij})$ which is defined to be
\Eq{
c_{i_l, j}&=c_{j, i_r}=\frac{d_ia_{ij}}{2},\tab j\in I\setminus\{i\},\\
c_{i_l,i_r}&=1.
}
\end{itemize}
The vertices are organized in \emph{levels}, such that the vertex $j\in I\setminus\{i\}$ is placed at level $j$, and $\{i_l,i_r\}$ are placed on the left and right hand side of level $i$ respectively. The arrows between the vertex $j$ and $\{i_l, i_r\}$ are dashed.

Intuitively, we call the set $(I\setminus\{i\})\cup \{i_l\}$ the \emph{left frozen vertices}, and $(I\setminus\{i\})\cup \{i_r\}$ the \emph{right frozen vertices}.
\end{Def}

\begin{Def}\label{basicquiver}Let $\bi=(i_1,...,i_m)$ be a reduced word. The \emph{basic quiver} $\bQ(\bi)$ is constructed by \emph{amalgamating} the elementary quivers 
\Eq{
\bQ(\bi):= \bJ(i_1)*\bJ(i_2)*\cdots*\bJ(i_m),\label{Qbi}
}
successively on each level, in the sense that the rightmost\footnote{By definition, on $\bJ(i)$, except the $i$-th level, the leftmost and rightmost frozen vertices coincide.} frozen vertices of $\bJ(i_{k-1})$ are glued to the leftmost frozen vertices of $\bJ(i_{k})$, with the weight of the corresponding arrows added together. Finally we remove all the vertices that are disjoint from the quiver, and redefine $Q_0$ so that the frozen vertices in the resulting quiver only consists of the left- and right-most vertices of each level.

We will write $\cX_q^{\bi}:=\cX_q^{\bQ(\bi)}$ and $\bT_q^{\bi}:=\bT_q^{\bQ(\bi)}$ to denote the quantum torus algebra associated to $\bQ(\bi)$ and its field of fractions respectively.
\end{Def}

\begin{Def}
Following \cite{Ip7}, we define the \emph{canonical labeling} of $\bQ(\bi)$ as
\Eq{
\{f_i^j \: i\in I, j=0,...,n_i\}
} where $n_i$ is the number of occurrences of $i$ in the reduced word $\bi$, such that the vertex at level $i$ and $j$-th position from the left is labeled by $f_i^j$.
\end{Def}

\begin{Ex} Consider $\g=\sl_4$ and let $\bi_0=(1,2,1,3,2,1)$ be a reduced word of $w_0\in W$. The basic quiver is then the amalgamation of 
\Eq{
\bQ(\bi_0)=\bJ(1)*\bJ(2)*\bJ(1)*\bJ(3)*\bJ(2)*\bJ(1)
}
see Figure \ref{basicq}.

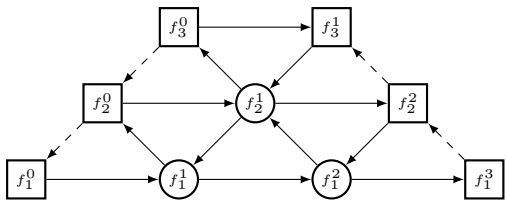
\begin{figure}[H]
\centering
\begin{tikzpicture}[every node/.style={inner sep=0, minimum size=0.5cm, thick, fill=white, draw}, x=1cm, y=1cm]
\node (1) at (0,0) {\tiny$f_1^0$};
\node (2) at (2,0) [circle]{\tiny$f_1^1$};
\node (3) at (4,0) [circle]{\tiny$f_1^2$};
\node (4) at (6,0)  {\tiny$f_1^3$};
\node (5) at (1,1) {\tiny$f_2^0$};
\node (6) at (3,1)  [circle]{\tiny$f_2^1$};
\node (7) at (5,1)  {\tiny$f_2^2$};
\node (8) at (2,2)  {\tiny$f_3^0$};
\node (9) at (4,2)  {\tiny$f_3^1$};
\drawpath{1,2,3,4}{}
\drawpath{5,6,7,3,6,2,5}{}
\drawpath{8,9,6,8}{}
\drawpath{8,5,1}{dashed}
\drawpath{4,7,9}{dashed}
\end{tikzpicture}
\caption{The basic quiver for $\bi_0=(1,2,1,3,2,1)$.}\label{basicq}
\end{figure}
\end{Ex}

Recall \cite[Section 7]{Ip7} that a Coxeter move $C:\bi\mapsto \bi'$ corresponds to a (sequence) of cluster mutations that transforms the quiver
$$\bQ(\bi)\to \bQ(\bi'),$$ and its pullback induces an isomorphism $\bT_q^{\bi'}\to \bT_q^{\bi}$.

To be more explicit, under the commutative Coxeter move $(...i,j...)\to (...j,i,...)$ the quiver $\bQ(\bi)=\bQ(\bi')$ stays the same. For the other two cases \eqref{Cox1}--\eqref{Cox2}, they correspond to the following mutations of the local pieces \eqref{Qbi} inside $\bQ(\bi)$ as
\Eq{\bJ(i)* \bJ(j)*\bJ(i)\to \bJ(j)*\bJ(i)*\bJ(j), \tab a_{ij}a_{ji}=1\label{quivermove1}}
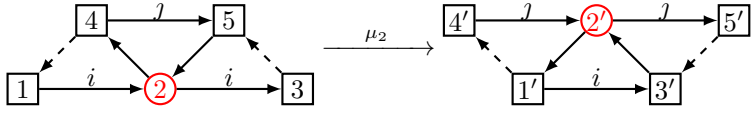
$$
\begin{tikzpicture}[every node/.style={inner sep=0, minimum size=0.4cm, thick}, baseline=(0), x=0.3cm, y=0.3cm]
\node (0) at (0,1.5){};
\node (1) at (0,0) [draw] {$1$};
\node (2) at (6,0) [red, draw, circle] {$2$};
\node (3) at (12,0) [draw] {$3$};
\node (4) at (3, 3) [draw]{$4$};
\node (5) at (9,3) [draw]{$5$};
\path (1) edge[->, thick] (2);
\node at (3,0.5) {$i$};
\path (2) edge[->, thick] (3);
\node at (9,0.5) {$i$};
\path (2) edge[->, thick] (4);
\path (5) edge[->, thick] (2);
\path (4) edge[->, thick] (5);
\node at (6,3.5) {$j$};
\path (4) edge[->, thick, dashed] (1);
\path (3) edge[->, thick, dashed] (5);
\end{tikzpicture}
\xto{\tab\mu_2\tab}
\begin{tikzpicture}[every node/.style={inner sep=0, minimum size=0.4cm, thick}, baseline=(0), x=0.3cm, y=0.3cm]
\node (0) at (0,1.5){};
\node (1) at (0,3) [draw] {$4'$};
\node (2) at (6,3) [draw, circle, red, thick] {$2'$};
\node (3) at (12,3) [draw] {$5'$};
\node (4) at (3, 0) [draw]{$1'$};
\node (5) at (9,0) [draw]{$3'$};
\path (1) edge[->, thick] (2);
\node at (3,3.5) {$j$};
\path (2) edge[->, thick] (3);
\node at (9,3.5) {$j$};
\path (2) edge[->, thick] (4);
\path (5) edge[->, thick] (2);
\path (4) edge[->, thick] (5);
\node at (6,0.5) {$i$};
\path (4) edge[->, thick, dashed] (1);
\path (3) edge[->, thick, dashed] (5);
\end{tikzpicture}
$$
and
\Eq{\bJ(i)* \bJ(j)*\bJ(i)* \bJ(j)\to \bJ(j)*\bJ(i)*\bJ(j)*\bJ(i), \tab a_{ij}a_{ji}=2\label{quivermove2}}
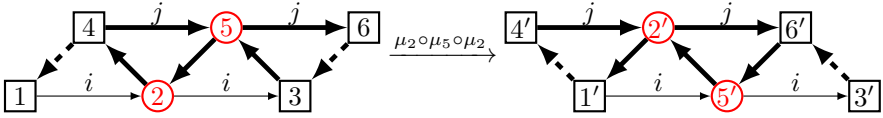
$$
\begin{tikzpicture}[every node/.style={inner sep=0, minimum size=0.4cm, thick}, baseline=(0), x=0.3cm, y=0.3cm]
\node (0) at (0,1.5){};
\node (1) at (0,0) [draw] {$1$};
\node (2) at (6,0) [red,draw, circle] {$2$};
\node (3) at (12,0) [draw] {$3$};
\node (4) at (3, 3) [draw]{$4$};
\node (5) at (9,3) [red,draw,circle]{$5$};
\node (6) at (15,3) [draw]{$6$};
\path (1) edge[->, thin] (2);
\path (2) edge[->, thin] (3);
\node at (3,0.6) {$i$};
\node at (9,0.6) {$i$};
\path (4) edge[->, vthick] (5);
\path (5) edge[->, vthick] (6);
\node at (6,3.6) {$j$};
\node at (12,3.6) {$j$};
\path (2) edge[->, vthick] (4);
\path (5) edge[->, vthick] (2);
\path (3) edge[->, vthick] (5);
\path (4) edge[->, vthick, dashed] (1);
\path (6) edge[->, vthick, dashed] (3);
\end{tikzpicture}
\xto{\mu_2\circ \mu_5\circ\mu_2}
\begin{tikzpicture}[every node/.style={inner sep=0, minimum size=0.4cm, thick},baseline=(0), x=0.3cm, y=0.3cm]
\node (0) at (0,1.5){};
\node (1) at (3,0) [draw] {$1'$};
\node (2) at (9,0) [red,draw, circle] {$5'$};
\node (3) at (15,0) [draw] {$3'$};
\node (4) at (0, 3) [draw]{$4'$};
\node (5) at (6,3) [draw,circle, red, thick]{$2'$};
\node (6) at (12,3) [draw]{$6'$};
\path (1) edge[->, thin] (2);
\path (4) edge[->, vthick] (5);
\path (5) edge[->, vthick] (6);
\path (6) edge[->, vthick] (2);
\path (2) edge[->, vthick] (5);
\path (5) edge[->, vthick] (1);
\path (2) edge[->, thin] (3);
\path (1) edge[->, vthick, dashed] (4);
\path (3) edge[->, vthick, dashed] (6);
\node at (6,0.6) {$i$};
\node at (12,0.6) {$i$};
\node at (3,3.6) {$j$};
\node at (9,3.6) {$j$};
\end{tikzpicture}
$$
where nodes on level $i$ are short.
\begin{Nota} Let $$\cC:=\{C_1\to C_2\to C_3\to\cdots\to C_k\} :\bi\mapsto \bi'$$ denotes a sequence of Coxeter moves, which corresponds to a sequence of cluster mutation $\mu_{i_m}\circ\cdots \circ\mu_{i_1}: \bQ(\bi)\to \bQ(\bi')$. For simplicity we will write these as
$$\mu_{\cC}:=\mu_{i_m}\circ\cdots\circ \mu_{i_1}:\bQ(\bi)\to \bQ(\bi')$$
$$\mu_{\cC}^q:=\mu_{i_1}^q\circ\cdots\circ \mu_{i_m}^q: \bT_q^{\bi'}\to \bT_q^{\bi}$$
\end{Nota}
%==================================================================================
\section{Quantum Lusztig's Transformation}\label{sec:qLT}%==================================================================================
\subsection{Rank 2 Case}\label{sec:qLT:rank2}
In \cite{Ip2}, we proposed a quantization of the Lusztig's transformation \eqref{121} in order to construct the unitary transformation necessary for the definition of positive representations of $\cU_q(\g_\R)$. It was stated as part of \cite[Proposition 5.1]{Ip2}, but we find it more logical to separately reproduce here as a Definition.

\begin{Def} Let $\C(q)\<\a,\b,\c\>$ be the noncommutative field of fractions generated by a triplet of quantum variables $(\a,\b,\c)$ satisfying
\Eq{
\a\b=q^2\b\a,\tab \c\a=q^2\a\c,\tab \b\c=\c\b.
}
Then we define $(\a',\b',\c')$ in $\C(q)\<\a,\b,\c\>$ by
\Eq{\label{qLustran}
\a'&:=(\a+\c)\inv \c\b =\b\c(\a+\c)\inv\\
\b'&:=\a+\c\nonumber\\
\c'&:=(\a+\c)\inv \a\b = \b\a(\a+\c)\inv\nonumber
}
such that they satisfy
\Eq{
\a'\b'=\b'\a',\tab \b'\c'=q^2\c'\b', \tab \c'\a'=q^2\a'\c'.
}
\end{Def}
The same formula \eqref{qLustran} applies to $(\a',\b',\c')$ shows that it is an involution. Indeed we can check that
\Eq{
(\a'',\b'',\c'')=(\a,\b,\c).
}

In Notation \ref{arrows}, they can be presented as\footnote{In \cite{Ip2}, we used the convention of opposite arrows}:
\Eq{\label{flip}\begin{tikzpicture}[baseline={([yshift=3ex]c.base)}]
\node (a) at (0, 0) {$\a$};
\node (b) at +(50: 1.17) {$\b$};
\node (c) at +(0: 1.5) {$\c$};
\foreach \from/\to in {a/c, b/a}
\draw [-angle 90] (\from) -- (\to);
\end{tikzpicture}
\tab\mbox{and}\tab
\begin{tikzpicture}[baseline={([yshift=3ex]b.base)}]
\node (a) at (0, 0) {$\a'$};
\node (b) at +(-50: 1.17) {$\b'$};
\node (c) at +(0: 1.5) {$\c'$};
\foreach \from/\to in {c/b, a/c}
\draw [-angle 90] (\from) -- (\to);
\end{tikzpicture}.}

In other words, \eqref{qLustran} defines an isomorphism
$$\C(q)\<a',\b',\c'\> \to \C(q)\<\a,\b,\c\>$$
and serves as the quantum analogue of the Lusztig's transformation \ref{121}. These collection of non-commutative variables and the corresponding reduced words are written \emph{formally} as
\Eq{\label{qLustranformal}
\bx_i(\a)\bx_j(\b)\bx_i(\c)=\bx_j(\a')\bx_i(\b')\bx_j(\c')
}

However, all these relations in fact suggest that the quantum Lusztig's transformation can be recast into the framework of quantum cluster mutation, which was not realized in \cite{Ip2}. Let 
\Eq{\label{abcX}
&X_3=\c,&& X_2=q\inv\a\c\inv,&& X_5=\b,\\
&X_5'=\c',&& X_2'=q\inv \a'{\c'}\inv,&& X_3'=\b'.\label{abcX2}
}
Then the $q$ commutation relation coincides with the subquivers in \eqref{quivermove1}, up to the arrows between the frozen nodes:
$$
\begin{tikzpicture}[every node/.style={inner sep=0, minimum size=0.4cm, thick}, baseline=(0), x=0.3cm, y=0.3cm]
\node (2) at (6,0) [red, draw, circle] {$2$};
\node (3) at (12,0) [draw] {$3$};
\node (5) at (9,3) [draw]{$5$};
\path (2) edge[->, thick] (3);
\path (5) edge[->, thick] (2);
\end{tikzpicture}
\xto{\tab\mu_2\tab}
\begin{tikzpicture}[every node/.style={inner sep=0, minimum size=0.4cm, thick}, baseline=(0), x=0.3cm, y=0.3cm]
\node (2) at (6,3) [draw, circle, red, thick] {$2'$};
\node (3) at (12,3) [draw] {$5'$};
\node (5) at (9,0) [draw]{$3'$};
\path (2) edge[->, thick] (3);
\path (5) edge[->, thick] (2);
\path (3) edge[->, thick] (5);
\end{tikzpicture}
$$

Furthermore, one can rewrite \eqref{qLustran} using \eqref{abcX} as
\Eq{
X_2'&=X_2\inv\\
X_3'&=X_3(1+qX_2)\\
X_5'&=X_5(1+qX_2\inv)\inv
}
which is exactly the quantum cluster mutation \eqref{qcm} at the vertex $2$! 

Using Notation \ref{xik}, note that \eqref{abcX}--\eqref{abcX2} is equivalent to
\Eq{
\a&=X_{2,3}, &\b&=X_{5}, &\c&=X_3,\\
\a'&=X_{2,3}',&\b'&=X_{5}',&\c'&=X_3',
}
so that the variables are manifestly self-adjoint under \eqref{star}.

Moreover, notice that the quantum cluster mutation formula does not depend on the adjacency between the frozen index ($3$ and $5$) at all, so in principle we are free to choose \emph{any} commutation relation between $\b$ and $\c$, which shows that the \emph{ad hoc} choices of \emph{orientation} arrows made in \cite[Definition 5.3]{Ip2} was actually unnecessary.

Now the doubly-laced case follows the same strategy using the Coxeter move \eqref{Cox2}. Unfortunately this was not discussed in \cite{Ip3}, so the following Definition completes its description. We define the analogue of \eqref{1212} as follows.

\begin{Def}Define the field of fractions $\C(q)\<\a,\b,\c,\d\>$ generated by $\{\a,\b,\c,\d\}$ with multipliers $\{1,2,1,2\}$ respectively, such that they satisfy the commutation relations given by the diagram
$$
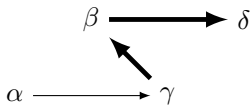
\begin{tikzpicture}
\node (a) at (0,0){$\a$};
\node (b) at (1,1) {$\b$};
\node (c) at (2,0){$\c$};
\node(d) at (3,1){$\d$};
\drawpath{c,b,d}{vthick}
\drawpath{a,c}{thin}
\end{tikzpicture} $$
i.e.
\Eq{
\a\c=q\inv\c\a,\tab \b\c=q^{-2}\b\c,\tab \b\d=q^2\d\b
}
and commute otherwise. Then we define $\{\a',\b',\c', \d'\}$ in $\C(q)\<\a,\b,\c,\d\>$ by
\Eq{\label{qLustran2}
\a'&:=SR\inv & \b'&:=\b\c^2\d S\inv\nonumber\\
\c'&:= \c\b\a R\inv &\d'&:= R^2 S\inv\nonumber
}
where (written non-commutatively)
\Eq{
R&:=\a\b+\a\d+\c\d,& S&:=\a^2\b+(\a+\c)^2\d.
}
\end{Def}

\begin{Prop} The transformation $(\a,\b,\c,\d)\mapsto (\a',\b',\c',\d')$ is an involution. The elements $R$ and $S$ commute, and the variables satisfy the commutation relation given by the diagram
$$
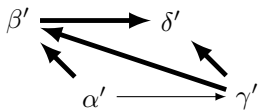
\begin{tikzpicture}
\node (a) at (0,1){$\b'$};
\node (b) at (1,0) {$\a'$};
\node (c) at (2,1){$\d'$};
\node(d) at (3,0){$\c'$};
\drawpath{b,a,c}{vthick}
\drawpath{d,c}{vthick}
\drawpath{d,a}{vthick}
\drawpath{b,d}{thin}
\end{tikzpicture} $$
\end{Prop}

This will serve as the quantum analogue of the Lusztig transform, again written \emph{formally} as
\Eq{\label{qLustranformal2}
\bx_i(\a)\bx_j(\b)\bx_i(\c)\bx_j(\d)=\bx_j(\b')\bx_i(\a')\bx_j(\d')\bx_i(\c').
}

We can again recast this as quantum cluster mutations (in the skew-symmetrizable case of \eqref{Cox2}) as follows.
\begin{Prop}
Let 
\Eq{
X_2&=q^{-\frac12}\a\c\inv,& X_3&=\c,& X_5&=q\inv\b\d\inv,& X_6&=\d, \label{X2356a}\\
X_2'&=q^{-\frac12}\a'{\c'}\inv,& X_3&=\c',& X_5&=q\inv \b'{\d'}\inv,& X_6&=\d' \label{X2356b}
}
Then we have the commutation relations given by the subquivers of \eqref{quivermove2} (up to arrows between frozen nodes)
$$
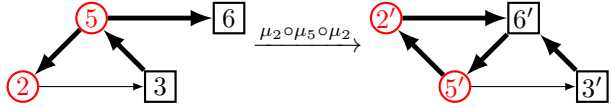
\begin{tikzpicture}[every node/.style={inner sep=0, minimum size=0.4cm, thick}, baseline=(0), x=0.3cm, y=0.3cm]
\node (0) at (0,1.5){};
\node (2) at (6,0) [red,draw, circle] {$2$};
\node (3) at (12,0) [draw] {$3$};
\node (5) at (9,3) [red,draw,circle]{$5$};
\node (6) at (15,3) [draw]{$6$};
\path (2) edge[->, thin] (3);
\path (5) edge[->, vthick] (6);
\path (5) edge[->, vthick] (2);
\path (3) edge[->, vthick] (5);
\end{tikzpicture}
\xto{\mu_2\circ \mu_5\circ\mu_2}
\begin{tikzpicture}[every node/.style={inner sep=0, minimum size=0.4cm, thick},baseline=(0), x=0.3cm, y=0.3cm]
\node (2) at (9,0) [red,draw, circle] {$5'$};
\node (3) at (15,0) [draw] {$3'$};
\node (5) at (6,3) [draw,circle, red, thick]{$2'$};
\node (6) at (12,3) [draw]{$6'$};
\path (5) edge[->, vthick] (6);
\path (6) edge[->, vthick] (2);
\path (2) edge[->, vthick] (5);
\path (2) edge[->, thin] (3);
\path (3) edge[->, vthick] (6);
\end{tikzpicture}
$$
and that 
\Eq{
(X_2',X_3',X_5',X_6')=\mu_2^q\mu_5^q\mu_2^q(X_2,X_3,X_5,X_6)
}
respectively.
\end{Prop}
Note that again \eqref{X2356a}--\eqref{X2356b} is equivalent to
\Eq{
\a&=X_{2,3}, &\b&=X_{5,6}, &\c&=X_3, &\d&=X_6\\
\a'&=X_{2,3}',&\b'&=X_{5,6}',&\c'&=X_3', &\d'&=X_6'
}
so that the variables are also self-adjoint under \eqref{star}.
%==================================================================================
\subsection{General Case}\label{sec:qLT:gen}
Armed with the above observations, we can now restate properly the construction of \cite{Ip2} in terms of quantum cluster mutations. 

Let $\bi_0=(i_1,...,i_N)\in\cI(w_0)$ be a longest reduced word, and consider the basic quiver $\bQ(\bi_0)$ constructed in Definition \ref{basicquiver} equipped with the canonical labeling.
\begin{Def} Let $\bH$ be an orientation of the Dynkin diagram with weights $\frac12 a_{ij}a_{ji}$ (presented as thin or thick dashed arrows with all the nodes frozen), and define the amalgamation
\Eq{
\til{\bQ}(\bi_0):=\bQ(\bi_0)*\bH
}
gluing $\bH$ from the right matching the decoration, so that all the dashed arrows between the rightmost frozen nodes of $\bQ(\bi_0)$ are either canceled or replaced by a solid one. 
\end{Def}
\begin{Def}
For $1\leq k\leq N$, if $i_k=i$, define $v_k$ such that $i_k$ is the $v_k$-th occurrence of $i$ from the left of $\bi_0$, and define the monomials $\a_k\in \cX_q^{\til{\bQ}(\bi_0)}$ by
\Eq{
\a_k:= X_{\sum_{j=v_k+1}^{n_i} f_i^j} \label{lusX}
}
i.e. it is (up to a $q$ factor) the product of all the $X$ variables on the $i$-th level of $\til{\bQ}(\bi_0)$ from the $k$-th position of $\bi_0$ to the rightmost frozen nodes.

Similarly define $\a_k'$ using the quantum cluster variables $X_i'$ of the quiver $\bQ(\bi_0')$ obtained from the mutation sequence corresponding to the Coxeter moves \eqref{quivermove1}--\eqref{quivermove2}.
\end{Def}
We can now restate the construction of the quantum Lusztig transformation, the proof of which follows directly from the explicit commutation relations between the $\a_k$'s defined in \cite[Definition 5.3]{Ip2}.
\begin{Prop}
$(\a_k)_{k=1}^N$ coincides with Definition 5.3 of \cite{Ip2} corresponding to the orientation $\bH$ of the Dynkin diagram. Furthermore, $(\a_k')_{k=1}^N$ coincides with the mutation rule Proposition 5.4 of \cite{Ip2}.
\end{Prop}
\begin{Ex} Consider type $A_3$ and an orientation of the Dynkin diagram by 
$$
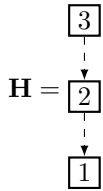
\bH=
\begin{tikzpicture}[baseline=(2),every node/.style={inner sep=0, minimum size=0.4cm, thick, fill=white, draw}, x=1cm, y=1cm]
\node (1) at (0,1) {$1$};
\node (2) at (0,2){$2$};
\node (3) at (0,3){$3$};
\drawpath{3,2,1}{dashed}
\end{tikzpicture}$$
If we label the quiver $\til{\bQ}(\bi_0)$ corresponding to the initial seed $\bi_0=(121321)$ given by
$$
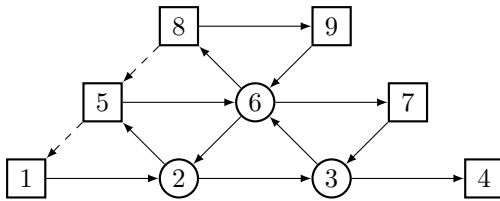
\begin{tikzpicture}[every node/.style={inner sep=0, minimum size=0.5cm, thick, fill=white, draw}, x=1cm, y=1cm]
\node (1) at (0,0) {$1$};
\node (2) at (2,0) [circle]{$2$};
\node (3) at (4,0) [circle]{$3$};
\node (4) at (6,0)  {$4$};
\node (5) at (1,1) {$5$};
\node (6) at (3,1)  [circle]{$6$};
\node (7) at (5,1)  {$7$};
\node (8) at (2,2)  {$8$};
\node (9) at (4,2)  {$9$};
\drawpath{1,2,3,4}{}
\drawpath{5,6,7,3,6,2,5}{}
\drawpath{8,9,6,8}{}
\drawpath{8,5,1}{dashed}
\end{tikzpicture}$$
then
$$\a_4=X_9,$$
$$\a_2=X_{6,7},\tab\a_5=X_7,$$
$$\a_1=X_{2,3,4},\tab\a_3=X_{3,4},\tab \a_6=X_4.$$
Notice that the leftmost frozen variables are not involved in the definition of $\a_k$.
\end{Ex}
\begin{Rem} Up to a choice of the orientation $\bH$ of the Dynkin diagram, we expect that the variables $(\a_k)_{k=1}^N$ can be identified with the canonical quantization of the \emph{primary coordinates} $P_i$ of \cite{GS} with a certain Poisson structure.
\end{Rem}
%==================================================================================
\subsection{Positive Representations}\label{sec:qLT:pos}
Given a longest reduced word $\bi_0$, we give the construction of the positive representation $\cP_\l^{\bi_0}$ of $\cU_q(\g_\R)$ in \cite{Ip2} in terms of self-adjoint operators induced from the transformation of the quantum Lusztig's coordinate defined in the previous section. Later in \cite{Ip7}, we implicitly understood the relationship between quantum Lusztig transformation and quantum cluster mutation, and gave the construction in terms of quantum cluster algebra using the basic quiver $\bQ(\bi^{op})*\bQ(\bi)$.

In both construction, the representation depends on the choice of $\bi_0$, but for any other longest reduced words $\bi_0'$, we have a unitary equivalence
\Eq{
\cP_\l^{\bi_0}\simeq \cP_\l^{\bi_0'}
}
since Coxeter moves correspond to a sequence of mutations, which under any polarization becomes a unitary transformation induced by the \emph{quantum dilogarithm function} \cite{Ip2}. 

Furthermore, the action of the Chevalley generators $\bf_i, \bK_i'$ can be written down explicitly as certain telescopic sums and monomials respectively, and is shown in both construction to depend only on the reduced word $\bi_0$, and coincide with the doubling of the \emph{Feigin's homomorphism} \cite{Re}. This was generalized to the Kac--Moody case in \cite{GS}.

However, the problem lies with the $\be_i, \bK_i$ generators. For each reduced words $\bi_0'$ ending with letter $i$, there exists a finite difference operator known as the \emph{canonical operator}. The action of the $\be_i, \bK_i$ generators for $\bi_0$ is then constructed by first taking any reduced word $\bi_0'$ which ends with the letter $i$, take the canonical operator in this special case, and transform this operator back to the given $\bi_0$ using a sequence of unitary transformations induced by the Coxeter moves. In the quantum cluster algebraic setting, the canonical operator is realized as a sum of two cluster monomials, and we perform the corresponding quantum cluster mutations in both the $\bQ(\bi)$ and $\bQ(\bi^{op})$ copy of the basic quiver.

Therefore the crucial step is to show that the construction of the representation of $\be_i, \bK_i$ does not depend on the choice of Coxeter moves from $\bi_0$ to $\bi_0'$. As pointed out in the Introduction, this statement is not explicit proved in both the simply-laced case \cite{Ip2} and the non-simply-laced case \cite{Ip3}, in which the cluster realization in \cite{Ip7} solely depends on.

With the discussion in the previous subsections, we can now state the Main Theorem as follows, the proof of which will be done in Section \ref{sec:mainPf} as a consequence of the Tits' Lemma.
\begin{Thm}\label{mainThm} Let $\bi_0, \bi_0'$ be two longest reduced word, and let $\cC_1, \cC_2:\bi_0\mapsto \bi_0'$ be two sequences of Coxeter moves. Let 
\Eq{
\mu_{\cC_j}^q: \bT_q^{\bi_0'}\to \bT_q^{\bi_0},\tab j=1,2
}
be the induced sequences of quantum cluster mutations corresponding to the Coxeter moves given by \eqref{Cox1}--\eqref{Cox2}. Then 
\Eq{\label{uCuC}
\mu_{\cC_1}^q=\mu_{\cC_2}^q.
}
\end{Thm}

Since the quantum Lusztig's coordinates are related to the quantum cluster variables by \eqref{lusX}, we can now state the quantum version of the injectivity Lemma as follows, which completes the gap of the construction of positive representations described in \cite{Ip2, Ip3}.
\begin{Cor}\label{qLT} If $\bi_0=(i_1,...,i_N)$ is a longest reduced word, and
\Eq{
\bx_{i_1}(a_1)\cdots \bx_{i_N}(a_N)=\bx_{i_1}(a_1')\cdots \bx_{i_N}(a_N')
}
 in the sense of \eqref{qLustranformal} or \eqref{qLustranformal2}, where the right hand side is obtained from the left hand side by a sequence of Coxeter moves $\bi_0\mapsto\cdots\mapsto \bi_0$ that returns to itself, then 
\Eq{
a_i'=a_i,\tab i=1,...,N
}
as elements in the noncommutative field of fractions generated by $\a_i$.
\end{Cor}
%==================================================================================
\section{Tits' Lemma}\label{sec:Tits}
In this section, we introduce the terminologies and preliminaries in order to state Tits' Lemma \cite{Tits}, which reduces the study of the fundamental groupoid of the reduced words graph under Coxeter moves, to cycles arising only from rank $3$ cases. This reduces the proof of Theorem \ref{mainThm} to explicitly checking only mutation sequences forming rank $3$ cycles. We will give a new, constructive proof of this result in Section \ref{sec:TitsPf}.
%==================================================================================
\subsection{Definitions}\label{sec:Tits:def}

\begin{Def}
A \emph{reduced word graph} of $w\in W$ is a graph $\cG:=\cG(w):=(\cG_0, \cG_1)$ where the vertex set $\cG_0:=\cI(w)$, and we have an edge joining $\bi,\bi'\in \cI(w)$ if the two words are related by a Coxeter move. We will use dashed edge to indicate move \eqref{Cox0}, and thick edges for \eqref{Cox2} or \eqref{Cox3}. It is well known that any two reduced word of $w\in W$ are related by a sequence of Coxeter moves, hence $\cG$ is a connected graph.
\end{Def}
In the construction of positive representations, Coxeter moves of the first type \eqref{Cox0} amounts to relabeling of the index of variables, and are not important in the algebraic structures of the representations. Therefore it is useful to consider a contraction of the reduced word graph and identify vertices that are related by the commutative Coxeter moves \eqref{Cox0}.

\begin{Def}
We define the equivalence classes $[\bi]$ of \emph{commutative words}, where $\bi\sim \bi'$ if they are related by a commutative Coxeter move \eqref{Cox0}. 
\end{Def}
\begin{Def}
A \emph{quotient reduced word graph} $\over{\cG}$ is a graph where the vertex set $\over{\cG}_0$ is the set of equivalence classes of commutative words $[\bi]$, and an edge joins $[\bi]$ and $[\bi']$ if there exists representatives $\bi\in[\bi]$ and $\bi'\in[\bi']$ that are related by a non-commutative Coxeter move \eqref{Cox1}--\eqref{Cox3}.
\end{Def}
In other words $\over{\cG}$ is obtained from a contraction of the reduced word graph $\cG$ where we contract any edges $\bi - \bi'$ corresponding to commutative Coxeter moves, and identify the vertices $\bi=\bi'$.

\begin{Nota}
We will use the \emph{stack notation} to denote the equivalence classes, such that commuting letter are written in the same position. For example in type $A_3$ consider $\bi=212321$, so that the position of $1,3$ can be interchanged by the commutative Coxeter move. We then use the notation
$$[\bi]=\{212321,213231,231213,231231\} :=: 2\sss2\sss. $$
\end{Nota}

\begin{Ex} If $\g$ is of rank $2$, for the longest element $w_0\in W$, obviously its reduced word graph just consists of a connected pair of vertices:

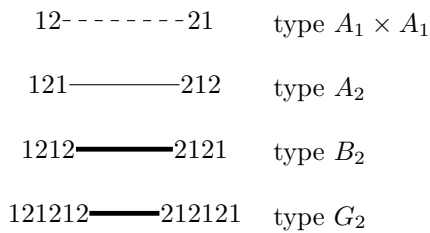
\begin{center}
\renewcommand{\arraystretch}{2}
\begin{tabular}{cl}
\begin{tikzpicture}[every node/.style={inner sep=0, minimum size=0.3cm, thick},  x=1cm, y=1cm]
\node (1) at (0,0){$12$};
\node (2) at (2, 0) {$21$};
\path (1) edge[-,dashed] (2);
\end{tikzpicture}& type $A_1\x A_1$\\
\begin{tikzpicture}[every node/.style={inner sep=0, minimum size=0.3cm, thick},  x=1cm, y=1cm]
\node (1) at (0,0){$121$};
\node (2) at (2, 0) {$212$};
\path (1) edge[-] (2);
\end{tikzpicture}& type $A_2$\\
\begin{tikzpicture}[every node/.style={inner sep=0, minimum size=0.3cm, thick},  x=1cm, y=1cm]
\node (1) at (0,0){$1212$};
\node (2) at (2, 0) {$2121$};
\path (1) edge[-, vthick] (2);
\end{tikzpicture}& type $B_2$\\
\begin{tikzpicture}[every node/.style={inner sep=0, minimum size=0.3cm, thick},  x=1cm, y=1cm]
\node (1) at (0,0){$121212$};
\node (2) at (2, 0) {$212121$};
\path (1) edge[-,vthick] (2);
\end{tikzpicture}& type $G_2$
\end{tabular}
\end{center}
\end{Ex}
\begin{Ex}\label{ExRank3} If $\g$ is decomposable type of rank $3$, i.e. of type $A_1\x A_1 \x A_1$, $A_2\x A_1$ or $B_2\x A_1$, then $\cG$ is a $6$-,$8$- or $10$-gon respectively, see Figure \ref{Rank3Graph}. 

\begin{figure}[H]
\centering
\begin{tikzpicture}[every node/.style={inner sep=0, minimum size=0.3cm, thick}, baseline=(0), x=2cm, y=1cm]
\node (1) at (0,3){$123$};
\node (2) at (0,2){$213$};
\node(3) at (0,1){$231$};
\node(4) at (1,1){$321$};
\node(5) at (1,2){$312$};
\node(6) at (1,3){$132$};
\node (0)at (0.5,0){$A_1\x A_1\x A_1$};
\path (1) edge[-,dashed] (2);
\path (2) edge[-,dashed] (3);
\path (3) edge[-,dashed] (4);
\path (4) edge[-,dashed] (5);
\path (5) edge[-,dashed] (6);
\path (6) edge[-,dashed] (1);
\end{tikzpicture}
\tab\tab
\begin{tikzpicture}[every node/.style={inner sep=0, minimum size=0.3cm, thick}, baseline=(0),  x=2cm, y=1cm]
\node (1) at (2,0){$1213$};
\node (2) at (1,0){$2123$};
\node(3) at (1,1){$2132$};
\node(4) at (1,2){$2312$};
\node(5) at (1,3){$3212$};
\node(6) at (2,3){$3121$};
\node(7) at (2,2){$1321$};
\node(8) at (2,1){$1231$};
\node (0) at (1.5,-1){$A_2\x A_1$};
\path (1) edge[-] (2);
\path (2) edge[-,dashed] (3);
\path (3) edge[-,dashed] (4);
\path (4) edge[-,dashed] (5);
\path (5) edge[-] (6);
\path (6) edge[-,dashed] (7);
\path (7) edge[-,dashed] (8);
\path (8) edge[-,dashed] (1);
\end{tikzpicture}
\tab\tab
\begin{tikzpicture}[every node/.style={inner sep=0, minimum size=0.3cm, thick}, baseline=(0),  x=2cm, y=1cm]
\node (1) at (2,0){$12123$};
\node (2) at (1,0){$21213$};
\node(3) at (1,1){$21231$};
\node(4) at (1,2){$21321$};
\node(5) at (1,3){$23121$};
\node(6) at (1,4){$32121$};
\node(7) at (2,4){$31212$};
\node(8) at (2,3){$13212$};
\node(9) at (2,2){$12312$};
\node(10) at (2,1){$12132$};
\node(0) at (1.5,-1){$B_2\x A_1$};
\path (1) edge[-, vthick] (2);
\path (2) edge[-,dashed] (3);
\path (3) edge[-,dashed] (4);
\path (4) edge[-,dashed] (5);
\path (5) edge[-, dashed] (6);
\path (6) edge[-,vthick] (7);
\path (7) edge[-,dashed] (8);
\path (8) edge[-,dashed] (9);
\path (9) edge[-,dashed] (10);
\path (10) edge[-,dashed] (1);
\end{tikzpicture}
\caption{Reduced word graph of decomposable rank $3$ types. Here the index $1,2,3$ correspond to the respective type factors indicated.\label{Rank3Graph}}
\end{figure}
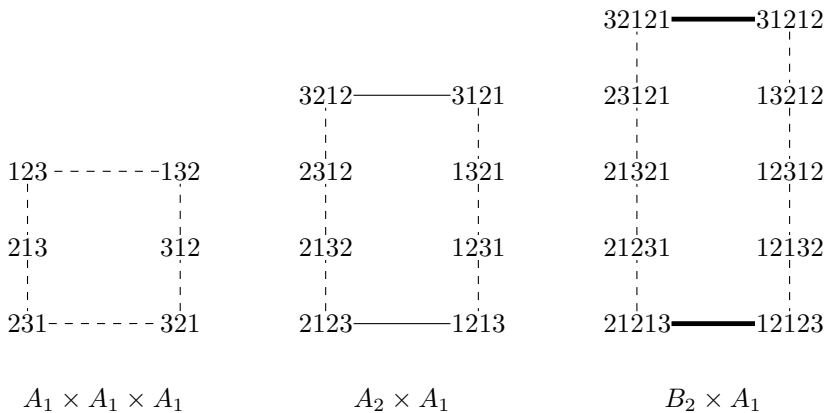
\end{Ex}

The most important cases we need in this paper are the simple rank $3$ cases.
\begin{Thm} \label{14gon} The quotient reduced word graph for the longest element $w_0\in W$ is a
\begin{itemize}
\item $8$-gon if $\g$ is of type $A_3$,
\item $14$-gon if $\g$ is of type $B_3$.
\end{itemize}
\end{Thm}
\begin{proof} This can be done by computing the reduced word graph $\cG$ explicitly. For type $A_3$ it is given by Figure \ref{A3graph} . The corresponding Coxeter moves among the vertices is clear by comparing the different expression of the reduced words.

\begin{figure}[H]
\centering
\begin{tikzpicture}[every node/.style={inner sep=0, minimum size=0.3cm, thick},  x=2cm, y=1cm]
\node (1a) at (3,7){$121321$};
\node (2) at (2,6){$212321$};
\node (3a) at (2,5){$213231$};
\node (3b) at (1,4){$213213$};
\node(3c) at (3,4){$231231$};
\node(3d) at (2,3){$231213$};
\node(4) at (2,2){$232123$};
\node(5a) at (3,1){$323123$};
\node(5b) at (4,1){$321323$};
\node(6) at (5,2){$321232$};
\node(7a) at (5,3){$312132$};
\node(7b) at (4,4){$132132$};
\node(7c) at (6,4){$312312$};
\node(7d) at (5,5){$132312$};
\node(8) at (5,6){$123212$};
\node (1b) at (4,7){$123121$};
\path (1a) edge[-] (2);
\path (2) edge[-] (3a);
\path (3d) edge[-] (4);
\path (4) edge[-] (5a);
\path (5b) edge[-] (6);
\path (6) edge[-] (7a);
\path (7d) edge[-] (8);
\path (8) edge[-] (1b);
\path (3a) edge[-,dashed] (3b);
\path (3b) edge[-,dashed] (3d);
\path (3d) edge[-,dashed] (3c);
\path (3c) edge[-,dashed] (3a);
\path (7a) edge[-,dashed] (7b);
\path (7b) edge[-,dashed] (7d);
\path (7d) edge[-,dashed] (7c);
\path (7c) edge[-,dashed] (7a);
\path (1a) edge[-,dashed] (1b);
\path (5a) edge[-,dashed] (5b);
\end{tikzpicture}
\caption{Reduced word graph in type $A_3$.\label{A3graph}}
\end{figure}
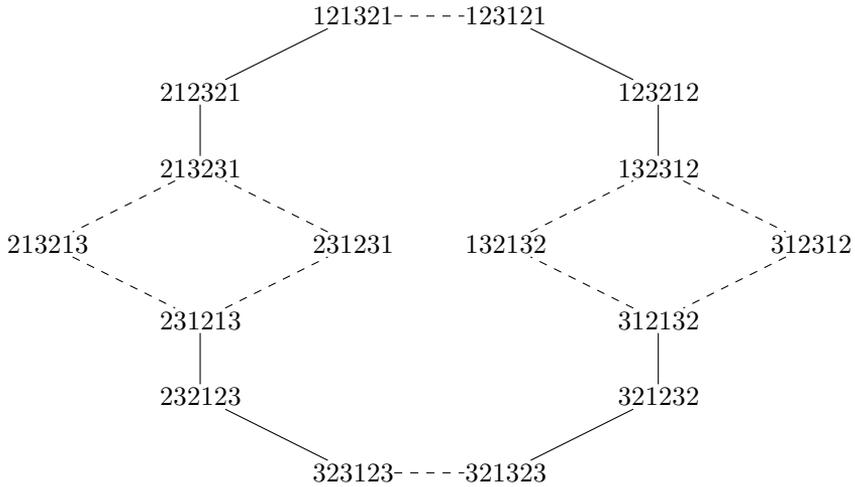
On the other hand the quotient graph $\over{\cG}$ forms an $8$-gon as in Figure \ref{A3quotient}
\begin{figure}[H]
\centering
\begin{tikzpicture}[every node/.style={inner sep=0, minimum size=0.3cm, thick},  x=2cm, y=1cm]
\node (1) at (3,5){$[\bi_0]=12\sss21$};
\node (2) at (2,4){$[\bi_1]=212321$};
\node (3) at (2,3){$[\bi_2]=2\sss2\sss$};
\node(4) at (2,2){$[\bi_3]=232123$};
\node(5) at (3,1){$[\bi_4]=32\sss23$};
\node(6) at (4,2){$[\bi_5]=321232$};
\node(7) at (4,3){$[\bi_6]=\sss2\sss2$};
\node(8) at (4,4){$[\bi_7]=123212$};
\path (1) edge[-] (2);
\path (2) edge[-] (3);
\path (3) edge[-] (4);
\path (4) edge[-] (5);
\path (5) edge[-] (6);
\path (6) edge[-] (7);
\path (7) edge[-] (8);
\path (8) edge[-] (1);
\end{tikzpicture}
\caption{Quotient reduced word graph in type $A_3$.\label{A3quotient}}
\end{figure}
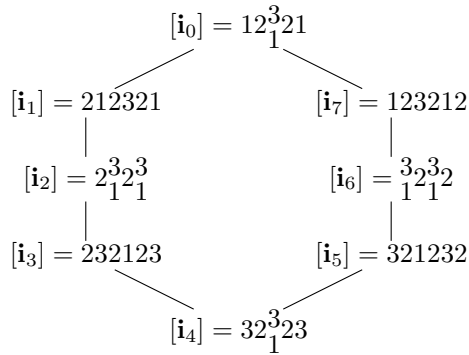

The quotient reduced word graph $\over{\cG}$ in type $B_3$ is given in Figure \ref{B3quotient}. Here $a_{12}a_{21}=2$ so that we need to consider Coxeter move \eqref{Cox2}: $1212- 2121$.
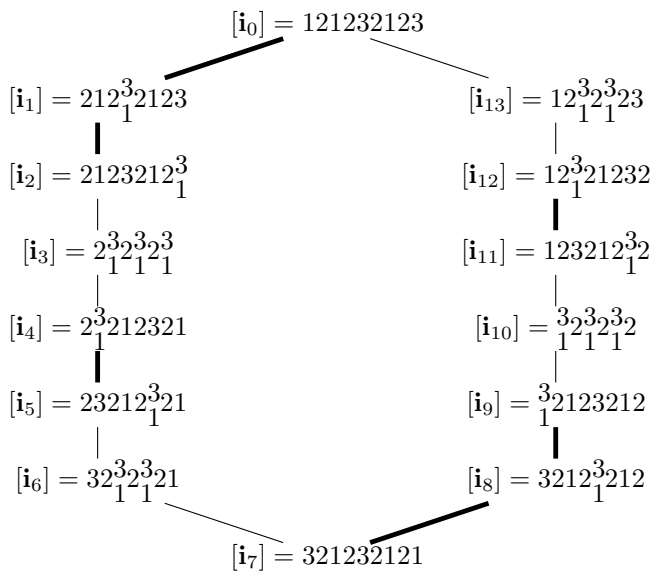
\begin{figure}[H]
\centering
\begin{tikzpicture}[every node/.style={inner sep=0, minimum size=0.3cm, thick},  x=3cm, y=1cm]
\node (1) at (2,6){$[\bi_{13}]=12\sss2\sss23$};
\node (2) at (1,7){$[\bi_0]=121232123$};
\node (3) at (0,6){$[\bi_1]=212\sss2123$};
\node(4) at (0,5){$[\bi_2]=2123212\sss$};
\node(5) at (0,4){$[\bi_3]=2\sss2\sss2\sss$};
\node(6) at (0,3){$[\bi_4]=2\sss212321$};
\node(7) at (0,2){$[\bi_5]=23212\sss21$};
\node(8) at (0,1){$[\bi_6]=32\sss2\sss21$};
\node(9) at (1,0){$[\bi_7]=321232121$};
\node(10) at (2,1){$[\bi_8]=3212\sss212$};
\node(11) at (2,2){$[\bi_9]=\sss2123212$};
\node(12) at (2,3){$[\bi_{10}]=\sss2\sss2\sss2$};
\node(13) at (2,4){$[\bi_{11}]=123212\sss2$};
\node(14) at (2,5){$[\bi_{12}]=12\sss21232$};
\path (1) edge[-] (2);
\path (2) edge[-, vthick] (3);
\path (3) edge[-, vthick] (4);
\path (4) edge[-] (5);
\path (5) edge[-] (6);
\path (6) edge[-, vthick] (7);
\path (7) edge[-] (8);
\path (8) edge[-] (9);
\path (9) edge[-, vthick] (10);
\path (10) edge[-, vthick] (11);
\path (11) edge[-] (12);
\path (12) edge[-] (13);
\path (13) edge[-, vthick] (14);
\path (14) edge[-] (1);
\end{tikzpicture}
\caption{Quotient reduced word graph in type $B_3$.\label{B3quotient}}
\end{figure}

The original reduced graph $\cG$ can be obtained from $\over{\cG}$ by replacing the commutative words with the corresponding subgraph that involves only the commutative Coxeter move \eqref{Cox0}. For example, in $\cG$ we have the subgraph as in Figure \ref{subgraphG}. In particular, while $\over{\cG}$ consists of $14$ vertices, $\cG$ consists of $42$ vertices.
\begin{figure}[H]
\centering
\begin{tikzpicture}[every node/.style={inner sep=0, minimum size=0.3cm, thick},  baseline=(2), x=1cm, y=0.9cm]
\node(0) at (0,3){$\vdots$};
\node (1) at (0,2){$2123212\sss$};
\node (2) at (0,1){$2\sss2\sss2\sss$};
\node(3) at (0,0){$2\sss212321$};
\node(4) at (0,-1){$\vdots$};
\path(0) edge[-,vthick](1);
\path (1) edge[-](2);
\path (2) edge[-](3);
\path(3) edge[-,vthick](4);
\end{tikzpicture}
\tab$\leadsto$\tab
\begin{tikzpicture}[every node/.style={inner sep=0, minimum size=0.3cm, thick}, baseline=(1), x=1cm, y=0.9cm]
\node (0) at (1,5){$\vdots$};
\node (0a) at (1,4){$212321213$};
\node (0b) at (5,4){$212321231$};
\node (1) at (0,1){$213213213$};
\node (2) at (1,3){$213231213$};
\node(3) at (3,2){$213231231$};
\node (4) at (2,0){$213213231$};
\node(5) at (4,1){$213213231$};
\node(6) at (5,3){$213231231$};
\node(7) at (7,2){$231231231$};
\node(8) at (6,0){$231213231$};
\node(9a) at (2,-1){$213212321$};
\node(9b) at (6,-1){$231212321$};
\node(9) at (6,-2){$\vdots$};
\path (1) edge[-, dashed] (2);
\path (2) edge[-, dashed] (3);
\path (3) edge[-, dashed] (4);
\path (4) edge[-, dashed] (1);
\path (5) edge[-, dashed] (6);
\path (6) edge[-, dashed] (7);
\path (7) edge[-, dashed] (8);
\path (8) edge[-, dashed] (5);
\path (1) edge[-, dashed] (5);
\path (2) edge[-, dashed] (6);
\path (3) edge[-, dashed] (7);
\path (4) edge[-, dashed] (8);
\path (0a) edge[-, dashed] (0b);
\path (9a) edge[-, dashed] (9b);
\path (0a) edge[-](2);
\path (0b) edge[-](6);
\path (9a) edge[-](4);
\path (9b) edge[-](8);
\path(0a) edge[-,vthick](0);
\path(9b) edge[-,vthick](9);
\end{tikzpicture}
\caption{A subgraph of $\cG$ obtained from ${\over{\cG}}$.\label{subgraphG}}
\end{figure}
\end{proof}

%==================================================================================
\subsection{Statement}\label{sec:Tits:state}
\begin{Def} \label{fundgp} Let $\cG$ be a connected graph. We say that the fundamental groupoid $\pi_1(\cG)$ is generated by closed paths $\a_1,...,\a_n$ if the fundamental group $\pi_1(\cG,*)$ over any fixed based points $*\in \cG_0$ is generated by closed paths of the form $\c\inv\circ \a_i\circ \c$ where $\c$ is any path joining $*$ to any vertex of $\a_i$.
\end{Def}

We can now state one of the main results that is required for the construction of positive representations as follows.
\begin{Thm}[Tits' Lemma]\label{Tits} Let $w\in W$ with $l(w)\geq 3$. The fundamental groupoid of the reduced word graph $\cG(w)$ is generated by closed paths of the form
\begin{itemize}
\item[(a)] Commuting Coxeter moves at different positions:
$$\begin{tikzpicture}[every node/.style={inner sep=0, minimum size=0.3cm, thick},  x=2cm, y=0.5cm]
\node (1) at (0,2){$\cdots \bp_{ij}\cdots \bp_{kl}\cdots$};
\node (2) at (2,2){$\cdots \bp_{ij}\cdots \bp_{lk}\cdots$};
\node (3) at (2,0){$\cdots \bp_{ji}\cdots \bp_{kl}\cdots$};
\node(4) at (0,0){$\cdots \bp_{ji}\cdots \bp_{lk}\cdots$};
\path (1) edge[-] (2);
\path (2) edge[-] (3);
\path (3) edge[-](4);
\path (4) edge[-] (1);
\end{tikzpicture}$$
\item[(b)] Rank 3 cycles:
$$\begin{tikzpicture}[every node/.style={inner sep=0, minimum size=0.3cm, thick},  x=1cm, y=0.5cm]
\node (1) at (0,2){$\bi'\bi_1\bi''$};
\node (2) at (2,2){$\cdots$};
\node (3) at (4,2){$\bi'\bi_r\bi''$};
\node (4) at (2,0){$\bi'\bi_0\bi''$};
\path (1) edge[-] (2);
\path (2) edge[-] (3);
\path (3) edge[-](4);
\path (4) edge[-] (1);
\end{tikzpicture}$$
where $\bi_0,...,\bi_r$ and their Coxeter moves involve only three elements of $I$.
\end{itemize}
\end{Thm}

By reducing the rank 3 cycles in type (b) to the quotient graph, we have
\begin{Cor} The fundamental groupoid $\pi_1(\over{\cG})$ of the quotient reduced word graph is generated by squares, octagons and 14-gons, corresponding to commuting Coxeter moves, and rank 3 Coxeter cycles of type $A_3$ and $B_3$ respectively.
\end{Cor}
\begin{proof}The only things to check is that for the rank $3$ cycles, after quotienting the commutative Coxeter moves, the index forms the subtypes
\begin{itemize}
\item $A_1\x A_1 \x A_1$: trivial
\item $A_2\x A_1$: single edge
\item $A_3$: octagon
\item $B_2\x A_1$ single edge
\item $B_3$: 14-gon
\end{itemize}
by Theorem \ref{14gon}.
\end{proof}

%==================================================================================
\section{Proof of Theorem \ref{Tits}}\label{sec:TitsPf}
We are ready to give a new, constructive proof of Tits' Lemma by producing an explicit decomposition of a given reduced word cycle. One can consider this argument as explicitly spelling out and simplifying the original abstract proof presented in \cite[Section 4]{Tits} using homotopies and morphisms of graphs.

\begin{Rem}Throughout the proof, all the closed paths are not assumed to be simple, but self intersection naturally breaks the paths into unions of smaller cycles, as well as trivial cycles of the form $\c\inv\circ\c$ in the fundamental groupoid $\pi_1(\cG)$. Therefore without loss of generality we may just present pictures of simple paths.
\end{Rem}

Let $\b\in \pi_1(\cG(w))$ be an oriented closed path in $\cG(w)$. We first make two elementary observations:
\begin{itemize}
\item[(*)] If there exists a path $\d$ cutting $\b$ into two components $\b_1, \b_2$, then $\b$ is generated by $\{\a_1,...,\a_n\}$ if both closed paths $\b_1\d$ and $\d\inv \b_2$ are generated by $\{\a_1,...,\a_n\}$ (in the sense of Definition \ref{fundgp}).
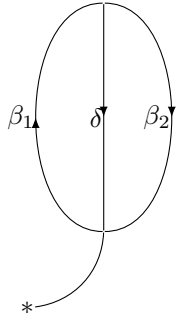
$$\begin{tikzpicture}[every node/.style={inner sep=0, minimum size=0cm, thick},  x=1cm, y=1cm]
\node (1) at (1,1){$*$};
\node (2) at (2,2){};
\node (3) at (2,5){};
\path(1) edge[-, bend right = 40] (2);
\path (3) edge[->-] node [left] {$\d$}(2);
\path (2) edge[->-, bend left =90] node[left]{$\b_1$}(3);
\path (3) edge[->-, bend left =90] node[left]{$\b_2$}(2);
\end{tikzpicture}$$
\item[(**)]  If $\b$ is a closed path such that all the words (vertices) share the same letter $i\in I$ in the last position, then $l(ws_i)<l(w)$, and $\b$ is isomorphic to a closed path $\over{\b}$ of $\cG(ws_i)$ where the vertex set consists of the same words with all the last letter removed. 

Conversely if $\over{\b}$ is a closed path in $\cG(w)$ and $l(ws_i)>l(w)$, then $\over{\b}$ is isomorphic to a closed path $\b$ in $\cG(ws_i)$ by adding the letter $i$ at the end of each word of $\over{\b}$.
\end{itemize}

Define the \emph{last letter map} $\w:\cI(w)\to I$ by 
\Eq{
\bi=(i_1,..,i_N)\mapsto i_N.} For any path $\b$ of $\cG$ we will also identify it as a subset of $\cG_0$ with its collection of vertices (words). Define \Eq{
m_i(\b)=\#\{\bi\in \b: \w(\bi)=i\}
}
to be the number of words in $\b$ ending with the letter $i$.

We are now ready to explain the proof of Theorem \ref{Tits} by induction. We will use a double induction on $n=l(w)$ and the number $m=|\w(\b)|$ of distinct last letter appearing in the words of $\b$.

For the base case of $n=3$, this is just a restatement of the type (b) cycles.

For the base case of $m=1$, every word $\bi\in\b$ is of the form $\bi=(..., i)$ for some letter $i\in I$. In particular $l(ws_i)<l(w)$ and hence by (**) the cycle is isomorphic to a cycle $\over{\b}$ in $\cG(ws_i)$, and we can use induction on $n$ to decompose $\over{\b}$, which by the converse of (**) also decompose $\b$.

Now assume $n>3$ and $m>1$. Fix an index $i\in \w(\b)$. Then by the definition of reduced word graph, there exists vertices $\bi_1, \bi_1, \bi_2', \bi_2'$ of $\b$ such that
\begin{itemize}
\item $\bi_1$ is of the form $(..., \bp_{ji})$ and $\bi_1'$ of the form $(..., \bp_{ij})$ for some letter $j\neq i$, and they are related by a single Coxeter move $C_{ij}$ involving the last letter.
\item $\bi_2$ is of the form $(..., \bp_{ki})$ and $\bi_2'$ of the form $(..., \bp_{ik})$ for some letter $k\neq i$, and they are related by a single Coxeter move $C_{ik}$ involving the last letter.
\item There exists a subpath $\b'$ of $\b$ joining $\bi_1$ and $\bi_2$, such that $\w(\bi)=i$ for any $\bi\in \b'$, see Figure \ref{bbpath}.
\end{itemize}
In particular, $l(ws_i)=l(ws_j)=l(ws_k)<l(w)$.
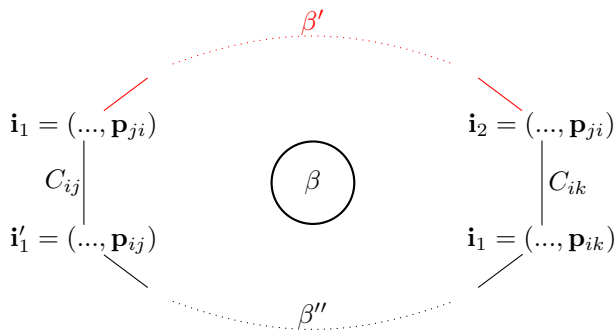
\begin{figure}[H]
\centering
\begin{tikzpicture}[every node/.style={inner sep=0, minimum size=0.3cm, thick},  x=2cm, y=1.5cm]
\node (0a) at (0.5, 1.5){};
\node (0b) at (2.5, 1.5){};
\node (1) at (0,1){$\bi_1=(..., \bp_{ji})$};
\node (2) at (3,1){$\bi_2=(..., \bp_{ji})$};
\node (3) at (0,0){$\bi_1'=(..., \bp_{ij})$};
\node(4) at (3,0){$\bi_1=(..., \bp_{ik})$};
\node(5a) at (0.5,-0.5){};
\node(5b) at (2.5,-0.5){};
\node at (1.5,0.5)[draw, circle,minimum size=0.6cm,inner sep=0.25cm]{$\b$};
\path(1) edge[-, red]  (0a);
\path(2) edge[-, red]  (0b);
\path(0a) edge[-, bend left = 20, dotted, red] node [above] {$\b'$} (0b);
\path(1) edge[-] node [left] {$C_{ij}$} (3);
\path(2) edge[-] node [right=2pt] {$C_{ik}$} (4);
\path(3) edge[-]  (5a);
\path(4) edge[-]  (5b);
\path(5a) edge[-, bend right = 20, dotted] node [above] {$\b''$} (5b);
\end{tikzpicture}
\caption{Decomposition of the cycle $\b$. All the vertices appearing in the red path end with the letter $i$.\label{bbpath}}
\end{figure}

If we let $J=\{i,j,k\}$ and consider the longest element $w_J\in W_J$, then by construction $l(ww_J)<l(w)$. Hence if $\over{\bi^k}\bp_{ij}$ is a reduced word for $w_J$, then there exists $\bj_j, \bj_j'\in\cI(w)$ which coincide except the last portion where they end differently with $\over{\bi^k}\bp_{ji}$ and $\over{\bi^k}\bp_{ij}$ respectively. Similarly for $\bj_k, \bj_k'$ which is the same as $\bj_j,\bj_j'$ except the last portion with $j\corr k$. 

Note that we allow $j=k$ in which case $\over{\bi^j}=\over{\bi^k}$ is empty, and $\bj_j,\bj_j'$ coincides with $\bj_k,\bj_k'$ respectively.

Since $\bi_1$ and $\bj_j$ both ends with $\bp_{ji}$, there exists a sequence of Coxeter moves $\cC:=\{C_l\}$ joining $\bi_1$ to $\bj_j$ which does not involve the $\bp_{ji}$ portion and commute with $C_{ij}$. The same sequence $\cC$ also transform $\bi_1'$ into $\bj_j'$, hence they form a sequence of commuting squares as in Figure \ref{comsq}. Similarly for $\bi_2, \bi_2'$ to $\bj_k$, $\bj_k'$.

\begin{figure}[H]
\centering
\begin{tikzpicture}[every node/.style={inner sep=0, minimum size=0.3cm, thick},  x=3cm, y=1.5cm]
\node (1) at (0,1){$\bi_1=(..., \bp_{ji})$};
\node (2) at (1,1){$(..., \bp_{ji})$};
\node (3) at (2,1){$\cdots$};
\node(4) at (3,1){$\bj_j=(..., \over{\bi_k}\bp_{ji})$};
\node (5) at (0,0){$\bi_1'=(..., \bp_{ij})$};
\node(6) at (1,0){$(..., \bp_{ij})$};
\node (7) at (2,0){$\cdots$};
\node(8) at (3,0){$\bj_j'=(..., \over{\bi_k}\bp_{ij})$};
\path(1) edge[-] node[above]{$C_1$}  (2);
\path(2) edge[-] node[above]{$C_2$}  (3);
\path(3) edge[-]  (4);
\path(5) edge[-] node[below]{$C_1$}  (6);
\path(6) edge[-] node[below]{$C_2$}  (7);
\path(7) edge[-]  (8);
\path(1) edge[-]  node[left] {$C_{ij}$}(5);
\path(2) edge[-]  node[left] {$C_{ij}$}(6);
\path(3) edge[-]  node[left] {$C_{ij}$}(7);
\path(4) edge[-]  node[right] {$C_{ij}$}(8);
\end{tikzpicture}
\caption{Commuting squares from $\bi_1,\bi_1'$ to $\bj_j, \bj_j'$. \label{comsq}}
\end{figure}

If $j\neq k$, by the explicit form in Example \ref{ExRank3} and Theorem \ref{14gon}, there exists a rank 3 cycle, that contains all $\bj_j$, $\bj_j'$, $\bj_k$, $\bj_k'$ in such a way that $\d$ joins $\bj_j$ and $\bj_k$ with $\w(\bi)=i$ for any $\bi\in \d$, while $\d'$ joins $\bj_j'$ and $\bj_k'$ with $\w(\bi)\neq i$ for any $\bi\in \d'$, see Figure \ref{bigcycle}. The cycle $C_{ij}\d C_{ik}\d'$ is then a rank 3 cycle of type (b). If $j=k$ then this part of the construction is trivial. 

The other cycle $\b_\#$ form by $\b'$ and $(\bi_1-\cdots-\bj_j -\overset{\d}{\cdots} - \bj_k-\cdots-\bi_2)$ all have words ending with the letter $i$, hence by (**) we can consider the last letter removed and apply induction on $n$ to decompose $\b_\#$. Therefore by (*), we are reduced to consider the remaining cycle $\b_*$ form by $\b''$ and $(\bi_1'-\cdots-\bj_j' -\overset{\d'}{\cdots} - \bj_k'-\cdots-\bi_2')$.

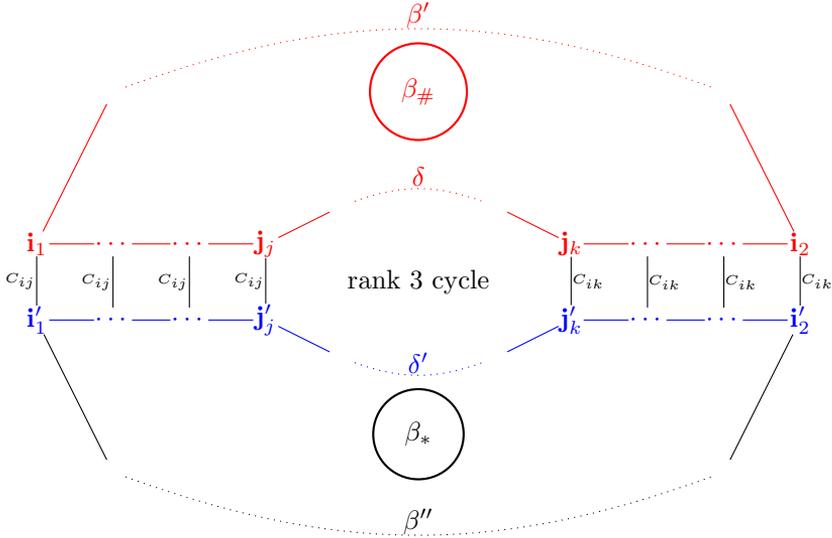
\begin{figure}[H]
\centering
\begin{tikzpicture}[every node/.style={inner sep=0, minimum size=0.3cm, thick},  x=1cm, y=1cm]
\node (0a) at (1, 4){};
\node (0b) at (9, 4){};
\node (1a) at (0,2)[red]{$\bi_1$};
\node (1b) at (1,2)[red]{$\cdots$};
\node (1c) at (2,2)[red]{$\cdots$};
\node (1d) at (3,2)[red]{$\bj_j$};
\node (2a) at (7,2)[red]{$\bj_k$};
\node (2b) at (8,2)[red]{$\cdots$};
\node (2c) at (9,2)[red]{$\cdots$};
\node (2d) at (10,2)[red]{$\bi_2$};
\node (3a) at (0,1)[blue]{$\bi_1'$};
\node (3b) at (1,1)[blue]{$\cdots$};
\node (3c) at (2,1)[blue]{$\cdots$};
\node (3d) at (3,1)[blue]{$\bj_j'$};
\node(4a) at (7,1)[blue]{$\bj_k'$};
\node(4b) at (8,1)[blue]{$\cdots$};
\node(4c) at (9,1)[blue]{$\cdots$};
\node(4d) at (10,1)[blue]{$\bi_2'$};
\node(5a) at (1,-1){};
\node(5b) at (9,-1){};
\node(6a) at (4,2.5){};
\node(6b) at (6,2.5){};
\node(7a) at (4,0.5){};
\node(7b) at (6,0.5){};
\node at (5,4)[draw, circle,minimum size=0.6cm,red, inner sep=0.25cm]{$\b_\#$};
\node at (5,-0.5)[draw, circle,minimum size=0.6cm,inner sep=0.25cm]{$\b_*$};
\node at (5,1.5){rank 3 cycle};
\path(1a) edge[-,red]  (0a);
\path(2d) edge[-,red]  (0b);
\path(1a) edge[-,red] (1b);
\path(1b) edge[-,red] (1c);
\path(1c) edge[-,red] (1d);
\path(2a) edge[-,red] (2b);
\path(2b) edge[-,red] (2c);
\path(2c) edge[-,red] (2d);
\path(3a) edge[-,blue] (3b);
\path(3b) edge[-,blue] (3c);
\path(3c) edge[-,blue] (3d);
\path(4a) edge[-,blue] (4b);
\path(4b) edge[-,blue] (4c);
\path(4c) edge[-,blue] (4d);
\path(1a) edge[-] node[left]{\tiny$C_{ij}$}(3a);
\path(1b) edge[-] node[left]{\tiny$C_{ij}$}(3b);
\path(1c) edge[-] node[left]{\tiny$C_{ij}$}(3c);
\path(1d) edge[-] node[left]{\tiny$C_{ij}$}(3d);
\path(2a) edge[-] node[right]{\tiny$C_{ik}$}(4a);
\path(2b) edge[-] node[right]{\tiny$C_{ik}$}(4b);
\path(2c) edge[-] node[right]{\tiny$C_{ik}$}(4c);
\path(2d) edge[-]  node[right]{\tiny$C_{ik}$}(4d);
\path(1d) edge[-,red] (6a);
\path(6b) edge[-,red] (2a);
\path(3d) edge[-,blue] (7a);
\path(7b) edge[-,blue] (4a);
\path(0a) edge[-, bend left = 20, dotted,red] node [above] {$\b'$} (0b);
\path(6a) edge[-, bend left = 20, dotted,red] node [above] {$\d$} (6b);
\path(3a) edge[-]  (5a);
\path(4d) edge[-]  (5b);
\path(5a) edge[-, bend right = 20, dotted] node [above] {$\b''$} (5b);
\path(7a) edge[-, bend right = 20, dotted,blue] node [above] {$\d'$} (7b);
\end{tikzpicture}
\caption{Decomposition of $\b$,where all the vertices along the red cycle end with the letter $i$, and all the vertices along the blue path does not end with letter $i$. \label{bigcycle}}
\end{figure}

Note that we now have $m_i(\b_*)<m_i(\b)$, hence by applying the same argument on $\b_*$ repeatedly, we reduce to the case where $i\notin\w(\b_*)$ so that $m$ decreases by at least $1$, and we can apply induction on $m$ to decompose $\b_*$ completely.

We remark that the $\bi''$ part of the rank 3 cycles in type (b) of Theorem \ref{Tits} comes from the letters removed during the induction step when $m=1$ and $n\mapsto n-1$. This conclude the proof of Theorem \ref{Tits}.

%==================================================================================
\section{Proof of Theorem \ref{mainThm}}\label{sec:mainPf}
We have the obvious property concerning operators defined over loops. Let $\{\cX^\bi\}_{\bi\in \cG_0}$ be a collection of algebras indexed by the vertices of $\cG$. Assume that for any oriented edge $\be: \bi_1\to \bi_2$ of $\cG$, we have a homomorphism $f_\be\in \mathrm{Hom}(\cX^{\bi_1},\cX^{\bi_2})$. For any oriented path $\a=(\be_1,...,\be_n)$ of $\cG$, let $f_\a:=f_{\be_n}\circ\cdots\circ f_{\be_1}$. 

\begin{Lem} Assume that $f_{\be\inv}\circ f_{\be}=\Id_{\cX_{\bi_1}}$ for any oriented edge $\be:\bi_1\to \bi_2$. If $\pi_1(\cG)$ is generated by closed paths $\a_1,...,\a_n$, and $f_{\a_i}=\Id$ for any $i$, then $f_\b=\Id$ for any closed path $\b$.
\end{Lem}
\begin{proof} We have
$$f_{\c\inv\a_i\c}=f_{\c\inv}\circ f_{\a_i}\circ f_{\c}=f_{\c\inv}\circ f_{\c} = \Id.$$ 
Since any closed path $\b$ is generated by $\c\inv\circ \a_i\circ \c$, the claims follow.
\end{proof}

Therefore to complete the proof of the Main Theorem \ref{mainThm}, we just need to check that the sequence of quantum cluster mutations along any commutative squares and any rank 3 cycles are trivial. The case of commutative squares is straightforward since it is just a change of index of the cluster variables. On the other hand, since the identity \eqref{uCuC} is satisfied for the classical relations due to Lusztig's Injectivity Lemma applied to the classical Bruhat cells of type $A_3$ or $B_3$, the quantum version is a consequence of the following result due to \cite{FG3}, see also \cite[Proposition 3.4]{KN}.

\begin{Thm}\label{Ymut}$(k_1,...,k_n)$ is a $\s$-period of the quantum $Y$-seed $(B,\bY)$ if and only if it is a $\s$-period of the classical $y$-seed $(B,y)$.
\end{Thm}

While this completes the proof of Theorem \ref{mainThm}, in order to be self-contained (and to avoid another mis-citation since we could not find a direct proof of this Theorem), we find it illuminating to provide an alternative, direct computational proof of the quantum identity \eqref{uCuC} over any 3-cycles.

%==================================================================================
\subsection{Algorithms for Explicit Checking}\label{sec:mainPf:check}
More precisely, fix any reduced word $\bi_0$ of $w_0$, and let 
\Eq{\mu_{i_m}\circ\cdots\circ \mu_{i_1}:\bQ(\bi_0)\to \bQ(\bi_0)
}
be a mutation sequence corresponding to a rank 3 cycle. Then we have to check that
\Eq{
\label{maineq}\mu_{i_1}^*\circ\cdots \circ \mu_{i_m}^*(X_i) = X_{\s(i)} \in \cX_q^{\bi_0}
}
where $\s$ is a permutation of index corresponding to the overall effects involving commutative Coxeter moves. 

In principle, the checking can be done explicitly, which involves \emph{huge} rational expressions in the quantum cluster variables of $\bT_q^{\bi_0}$ and it is very hard to present. More precisely, during the sequence of cluster mutations, the initial variable $X_i\in \cX_q^{\bi_0}$ will be mutated to a series of rational expressions, followed by factorization of the numerators and denominators where some factors canceled and the expression simplified, finally ending again at a simple monomial $X_{\s(i)}$. While it is OK to write down all such expressions explicitly in type $A_3$, the expressions become \emph{very complicated} in type $B_3$, and the factorization of non-commutative rational expressions could not be done effectively by simple computer program with symbolic computation. 

To aid with the presentation in order to avoid putting the calculations into a huge appendix, we employ several tricks to simplify our verification. These tricks will also shed light onto some general principles in calculation involving quantum cluster mutations.

\textbf{Trick 1:} To aid the bookkeeping of index, while \eqref{maineq} is obtained from pullback of the mutation sequence so that the order of operators are reversed, we can make use of the fact that $\mu_k^*$ are involution, hence by taking inverse we can deduce that \eqref{maineq} holds if and only if the forward version hold:
\Eq{
\label{id}\mu_{i_m}^*\circ\cdots \circ \mu_{i_1}^*(X_i) = X_{\s\inv (i)},\tab \forall X_i\in\cX_q^{\bi_0}
}
or in other words
\Eq{
\label{id2} \s^*\circ\mu_{i_m}^*\circ\cdots \circ \mu_{i_1}^*(X_i) = \Id_{\cX_q^{\bi_0}}
}
which is the identity that we will be checking. Also recall that if we use the \emph{canonical labeling} of the quiver $\bQ(\bi_0)$, then $\s^*$ becomes the identity map on $\cX_q^{\bi_0}$.

\textbf{Trick 2:} Let $P(X_{i_1},...,X_{i_k})$ be a Laurent polynomial which is linear in $X_{i_k}$ or $X_{i_k}\inv$. If the identity \eqref{id2} holds for $X_{i_1},...,X_{i_{k-1}}$ and $P(X_{i_1},...,X_{i_{k}})$, then the identity holds for $X_{i_{k}}$ since it is a rational expression of the former and \eqref{id2} is a homomorphism of $\bT_q^{\bi_0}$. In particular, the verification becomes very simple if the polynomial $P$ arises from quantum group embedding $\fD_q\inj \cX_q^{\bi_0}$ as in \cite{Ip7}.

\textbf{Trick 3:} If $X_\l$ is a \emph{standard monomial}, i.e. on $\L_\bQ$ we have
\Eq{\label{stdmon}
(\l, \vec{e_k})\leq 0,\tab \forall k\in Q\setminus Q_0,
}
then its image is universally Laurent under any mutation sequence \cite{GS}, and the calculation can be checked (by hand) and presented effectively.

We remark that even though Trick 3 utilizes the properties of standard monomials as universally Laurent polynomials, the calculation is just conveniently chosen to simplify the expression, and is independent of the proofs of those properties presented in \cite{GS}.

%==================================================================================
\subsection{Type $A_3$}\label{sec:mainPf:A3}
Let $\g$ be type $A_3$. We will use the labeling in Figure \ref{A3quotient} for our seeds, so that the cluster variables $X_i^{(k)}\in \cX_q^{\bi_k}$. (We will omit the bracket for $\bi_0$). Recall that if $\bi\sim \bi'$ in the commutative class, then the quiver $\bQ(\bi)=\bQ(\bi')$ up to relabeling the index.

The quiver corresponding to the initial seed $\bi_0=(121321)$ is labeled as in Figure \ref{A3quiver}.
\begin{figure}[H]
\centering
\begin{tikzpicture}[every node/.style={inner sep=0, minimum size=0.5cm, thick, fill=white, draw}, x=1cm, y=1cm]
\node (1) at (0,0) {$1$};
\node (2) at (2,0) [circle]{$2$};
\node (3) at (4,0) [circle]{$3$};
\node (4) at (6,0)  {$4$};
\node (5) at (1,1) {$5$};
\node (6) at (3,1)  [circle]{$6$};
\node (7) at (5,1)  {$7$};
\node (8) at (2,2)  {$8$};
\node (9) at (4,2)  {$9$};
\drawpath{1,2,3,4}{}
\drawpath{5,6,7,3,6,2,5}{}
\drawpath{8,9,6,8}{}
\drawpath{8,5,1}{dashed}
\drawpath{4,7,9}{dashed}
\end{tikzpicture}
\caption{Basic quiver in type $A_3$.\label{A3quiver}}
\end{figure}
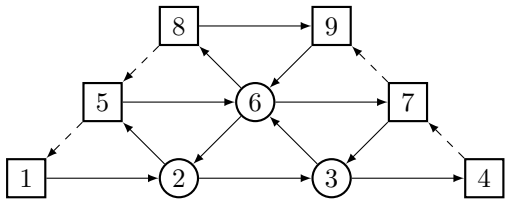
Then the rank 3 cycles corresponds to mutation at the vertices $2,6,3,2,6,3,2,6$ sequentially, and we have
\Eq{
\mu_6\circ\mu_2\circ\mu_3\circ\mu_6\circ\mu_2\circ\mu_3\circ\mu_6\circ\mu_2=\s
}
where $\s=(263)$ is the permutation changing the index $(2,3,6)$ into $(6,2,3)$ in the final quiver. 

\begin{Rem}\label{finite} In fact the mutable part of $Q(\bi_0)$ is cluster finite $A_3$ type, hence we know \emph{a priori} that there are finitely many (quantum) cluster variables, and the periodicity of the mutation sequence is automatic from the general theory of quantum cluster algebra. Nonetheless, we will give a direct computational proof of the identity.
\end{Rem}

Recall \cite[Theorem 4.14]{Ip7} that the image of the Chevalley generators of the lower Borel subalgebra $\cU_q(\fb_-)$ under the Coxeter moves depend only the reduced words $\bi\in\cI(w_0)$ and not on any choice of mutations. The image of the generators $\bf_i$ are given by a telescopic sum along the horizontal arrows of the basic quivers $\bQ(\bi)$ (with the right-most vertices removed), while $\bK_i'$ are the corresponding cluster monomials. Hence the identity \eqref{id2} holds automatically for these expressions. Explicitly they are given by:
\Eq{
&\bf_1=X_1+X_{1,2}+X_{1,2,3},&&\bK_1'=X_{1,2,3,4}\\
& \bf_2=X_5+X_{5,6},&& \bK_2'=X_{5,6,7}\nonumber\\
&\bf_3=X_8,&&\bK_3'=X_{8,9}\nonumber.
}
Also we observe that the monomials
\Eq{
C_1=X_{2,3,6},\tab C_2=X_{1,5,8}
}
commute with all the mutable variables $\{X_2,X_3,X_6\}$ and invariant under $\s$, therefore by Proposition \ref{propmon} they remain monomials throughout the sequence of cluster mutations and the identity \eqref{id} holds automatically.

By Trick 3, we claim that it is enough to check the identity \eqref{id} for the following \emph{standard monomial} $X_{5,8}$:

\Eqn{
X_{5,8}&\mapsto_{\mu_2^*}X_{5,8}^{(1)}+X_{2,5,8}^{(1)}\\
&\mapsto_{\mu_6^*}X_{5,8}^{(2)}+X_{5,6,8}^{(2)}+X_{2,5,6,8}^{(2)}\\
&\mapsto_{\mu_3^*}X_{5,8}^{(3)}+X_{5,6,8}^{(3)}+X_{2,5,6,8}^{(3)}+X_{2,3,5,6,8}^{(3)}\\
&\mapsto_{\mu_2^*}X_{2,5,8}^{(4)}+X_{2,5,6,8}^{(4)}+X_{2,3,5,6,8}^{(4)}\\
&\mapsto_{\mu_6^*}X_{2,5,8}^{(5)}+X_{2,3,5,8}^{(5)}\\
&\mapsto_{\mu_3^*}X_{2,5,8}^{(6)}\\
&\mapsto_{\mu_2^*}X_{5,8}^{(7)}\\
&\mapsto_{\mu_6^*}X_{5,8}.
}

Now it follows by Trick 2 that in the following, if \eqref{id} holds for the elements on the left hand side, then it also holds for the right hand side due to linearity:
\Eqn{
\bf_3&\imply X_8\\
\bK_3' &\imply X_9\\
X_{5,8} &\imply X_5\\
\bf_2&\imply X_{5,6}\imply X_6\\
\bK_2'&\imply X_7\\
C_2&\imply X_1\\
C_1&\imply X_{2,3}\imply X_{1,2,3}\\
X_{1,2,3}, \bf_1&\imply X_1+X_{1,2}\imply X_{1,2}\imply X_2\\
X_{2,3}&\imply X_3\\
\bK_1'&\imply X_4.
}

Hence we conclude that \eqref{id} holds for all the cluster variables $X_i\in\cX_q^{\bi_0}$. This completes the discussion in type $A_3$.

%==================================================================================
\subsection{Type $B_3$}\label{sec:mainPf:B3}
The situation is much more complicated for type $B_3$, and as explained, huge non-commutative rational expressions appear if we just consider the images of any cluster variables.

To begin, we will choose the initial word $\bi_0=(121232123)$ with the corresponding quiver given in Figure \ref{B3quiver}.

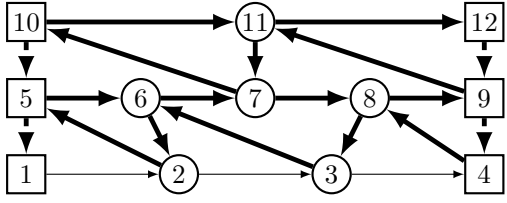
\begin{figure}[H]
\centering
\begin{tikzpicture}[every node/.style={inner sep=0, minimum size=0.5cm, thick, fill=white, draw}, x=0.5cm, y=1cm]
\node (1) at (0,0) {$1$};
\node (2) at (4,0) [circle]{$2$};
\node (3) at (8,0) [circle]{$3$};
\node (4) at (12,0) {$4$};
\node (5) at (0,1) {$5$};
\node (6) at (3,1) [circle]{$6$};
\node (7) at (6,1) [circle]{$7$};
\node (8) at (9,1) [circle]{$8$};
\node (9) at (12,1) {$9$};
\node (10) at (0,2) {$10$};
\node (11) at (6,2) [circle]{$11$};
\node (12) at (12,2) {$12$};
\drawpath{1,2,3,4}{}
\drawpath{4,8,3,6,2,5}{vthick}
\drawpath{5,6,7,8,9}{vthick}
\drawpath{9,11,7,10}{vthick}
\drawpath{10,11,12}{vthick}
\drawpath{10,5,1}{dashed,vthick}
\drawpath{12,9,4}{dashed,vthick}
\end{tikzpicture}
\caption{Basic quiver in type $B_3$.\label{B3quiver}}
\end{figure}

\begin{Rem} The mutable part of this quiver is of finite-mutation type $F_4^{(*,*)}$, but it is not finite type \cite[Table 6]{FST}. In particular there will be infinitely many cluster variables, therefore the same argument in Remark \ref{finite} of type $A_3$ using finitely many cluster variables do not apply in this case.
\end{Rem}

The rank 3 cycle of Coxeter moves corresponds to a sequence of 26 mutations from $\bi_0$ given sequentially by
\Eq{
\underline{2,6,2},\underline{3,8,3},7,11,\underline{2,8,2},6,7,\underline{3,11,3},\underline{2,8,2},7,6,\underline{3,8,3},11,7.
}
Here the cluster mutations of the underlined triplets correspond to the Coxeter moves involving the short index as in \eqref{quivermove2}. Furthermore, in type $B_3$, it turns out that we do not need any extra permutation $\s$ to complete the cycle.

To check the identity \eqref{id}, again it is satisfied automatically by some special elements of $\cX_q^{\bi_0}$, including the image of the generators of $\cU_q(\fb_-)$ expressed by the cluster variables along the horizontal arrows of $\bQ(\bi)$:

\Eq{
&\bf_1=X_1+X_{1,2}+X_{1,2,3},&&\bK_1'=X_{1,2,3,4},\\
&\bf_2=X_5+X_{5,6}+X_{5,6,7}+X_{5,6,7,8},&&\bK_2'=X_{5,6,7,8,9},\nonumber\\
&\bf_3=X_{10}+X_{10,11},&&\bK_3'=X_{10,11,12},\nonumber
}
as well as some monomials that commute with all the mutable variables:
\Eq{
C_1=X_{2^2,7^{-2},8^{-1},9^{-2},11^{-1}}, && C_2=X_{3^2,6,7^2,8,12}, && C_3=X_{4^2,8,9^2,11,12^2},
}
which remain monomials during the sequence of cluster mutations and satisfy \eqref{id} automatically by Proposition \ref{propmon}.

We claim that it suffices to check that the identity \eqref{id} holds for the following monomials in $\cX_q^{\bi_0}$.
\Eq{
\label{4element}
X_{1}\inv,\tab X_{12}, \tab X_{9,12},\tab X_{4,9,12},
}
from which the following implications follow by Trick 2 due to linearity:
\Eqn{
X_1\inv&\imply X_1\\
X_{12}, X_{9,12}&\imply X_9\\
X_{4,9,12} &\imply X_4\\
\bK_1'&\imply X_{1,2,3}\\
X_{1,2,3}, \bf_1&\imply X_1+X_{1,2}\imply X_2 \imply X_3\\
X_{12},\bK_3'&\imply X_{10,11}\\
\bf_3, X_{10,11}&\imply X_{10}\imply X_{11}\\
C_3&\imply X_8\\
C_1&\imply X_7\\
C_2&\imply X_6\\
\bf_2&\imply X_5.
}
Hence we conclude the identity \eqref{id} holds for all the cluster variables $X_i\in\cX_q^{\bi_0}$. 

It remains to check the identity \eqref{id} holds for the 4 elements in \eqref{4element}, which are actually \emph{standard monomials} in $\cX_q^{\bi_0}$ by checking \eqref{stdmon} directly. Hence by Trick 3, they are \emph{universally Laurent polynomials} and their mutations can be reasonably computed. In the following, to simplify notations we indicate with the bracket $(k)$ which seed $\cX_q^{\bi_k}$ the elements on the right belong to, following the indexing in Figure \ref{B3quotient}. Also 
we denote by $[2]:=q^{\frac12}+q^{-\frac12}$.

{\tiny\Eqn{
X_1\inv&\mapsto_{\mu_2^*\mu_6^*\mu_2^*}&& X_1\inv+X_{1,6}\inv+X_{1,2,6}\inv&(1)\\
&\mapsto_{\mu_3^*\mu_8^*\mu_3^*}&& X_1\inv+X_{1,6}\inv+X_{1,2,6}\inv+X_{1,2,6,8}\inv+X_{1,2,3,6,8}\inv&(2)\\
&\mapsto_{\mu_7^*}&& X_1\inv+X_{1,6}\inv+X_{1,2,6}\inv+X_{1,6,7}\inv+X_{1,2,6,7}\inv+X_{1,2,6,7,8}\inv+X_{1,2,3,6,7,8}\inv&(3)\\
&\mapsto_{\mu_{11}^*}&&X_{1}\inv+X_{1,6}\inv+X_{1,2,6}\inv+X_{1,6,7}\inv+X_{1,2,6,7}\inv+X_{1,6,7,11}\inv+X_{1,2,6,7,11}\inv+X_{1,2,6,7,8,11}\inv+X_{1,2,3,6,7,8,11}\inv&(4)\\
&\mapsto_{\mu_2^*\mu_8^*\mu_2^*}&&X_1\inv+X_{1,8}\inv+X_{1,2,8}\inv+X_{1,6,8}\inv+X_{1,2,6,8}\inv+X_{1,6,7,8}\inv+X_{1,2,6,7,8}\inv+X_{1,2,6,7,8,11}\inv+X_{1,2,3,6,7,8,11}\inv&(5)\\
&\mapsto_{\mu_6^*}&& X_1\inv+X_{1,8}\inv+X_{1,2,8}\inv+X_{1,7,8}\inv+X_{1,2,7,8}\inv+X_{1,2,7,8,11}\inv+X_{1,2,3,7,8,11}\inv&(6)\\
&\mapsto_{\mu_7^*}&& X_1\inv+X_{1,8}\inv+X_{1,2,8}\inv+X_{1,2,8,11}\inv+X_{1,2,3,8,11}\inv&(7)\\
&\mapsto_{\mu_3^*\mu_{11}^*\mu_3^*}&&X_1\inv+X_{1,8}\inv+X_{1,2,8}\inv&(8)\\
&\mapsto_{\mu_2^*\mu_8^*\mu_2^*}&&X_1\inv&(9)\\
&\mapsto_{\mu_7^*}&&X_1\inv&(10)\\
&\mapsto_{\mu_6^*}&&X_1\inv&(11)\\
&\mapsto_{\mu_3^*\mu_8^*\mu_3^*}&&X_1\inv&(12)\\
&\mapsto_{\mu_{11}^*}&&X_1\inv&(13)\\
&\mapsto_{\mu_7^*}&&X_1\inv\\
}
}
{\tiny
\Eqn{
X_{12}&\mapsto_{\mu_2^*\mu_6^*\mu_2^*}&&X_{12}&(1)\\
&\mapsto_{\mu_3^*\mu_8^*\mu_3^*}&&X_{12}&(2)\\
&\mapsto_{\mu_7^*}&& X_{12}&(3)\\
&\mapsto_{\mu_{11}^*}&&X_{12}+X_{11,12}&(4)\\
&\mapsto_{\mu_2^*\mu_8^*\mu_2^*}&&X_{12} +X_{11,12}+[2]X_{2,11,12}+X_{2^2,11,12}+X_{2^2,8,11,12}&(5)\\
&\mapsto_{\mu_6^*}&& X_{12}+X_{11,12}+[2]X_{2,11,12}+X_{2^2,11,12}+X_{2^2,8,11,12}+X_{2^2,6,8,11,12}&(6)\\
&\mapsto_{\mu_7^*}&& X_{12} + X_{7,12} + X_{7,11,12} + [2] X_{2,7,11,12} +  X_{2^2,7,11,12} + X_{2^2,7,8,11,12} +  X_{2^2,7^2,8,11,12}&(7)\\
&&& + X_{2^2,6,7^2,8,11,12}\\
&\mapsto_{\mu_3^*\mu_{11}^*\mu_3^*}&&X_{12} + X_{7,12} + [2] X_{2,7,12} + X_{2^2,7,12} +  X_{2^2,7,8,12} + X_{7,11,12}+ [2] X_{2,7,11,12} +  X_{2^2,7,11,12} &(8)\\
&&& + [2] X_{2,3,7,11,12}+  [2] X_{2^2,3,7,11,12}+ X_{2^2,3^2,7,11,12} +  X_{2^2,7,8,11,12}+ [2] X_{2^2,3,7,8,11,12} \\
&&& +  X_{2^2,3^2,7,8,11,12} +  X_{2^2,3^2,7^2,8,11,12} +  X_{2^2,3^2,6,7^2,8,11,12}\\
&\mapsto_{\mu_2^*\mu_8^*\mu_2^*}&&X_{12} + X_{7,12} + X_{7,11,12} + [2] X_{3,7,11,12} +  X_{3,7,11,12} + X_{3,7,8,11,12} +  X_{3,7,8,11,12} + X_{3,6,7,8,11,12}&(9)\\
&\mapsto_{\mu_7^*}&&X_{12}+X_{11,12}+[2]X_{3,11,12}+X_{3^2,11,12}+X_{3^2,8,11,12}+X_{3^2,6,8,11,12}&(10)\\
&\mapsto_{\mu_6^*}&&X_{12}+X_{11,12}+[2]X_{3,11,12}+X_{3^2,11,12}+X_{3^2,8,11,12}&(11)\\
&\mapsto_{\mu_3^*\mu_8^*\mu_3^*}&&X_{12}+X_{11,12}&(12)\\
&\mapsto_{\mu_{11}^*}&&X_{12}&(13)\\
&\mapsto_{\mu_7^*}&&X_{12}
}
}
{\tiny
\Eqn{
X_{9,12}&\mapsto_{\mu_2^*\mu_6^*\mu_2^*}&&X_{9,12}&(1)\\
&\mapsto_{\mu_3^*\mu_8^*\mu_3^*}&&X_{9,12}+[2]X_{3,9,12}+X_{3^2,9,12}+X_{3^2,8,9,12}&(2)\\
&\mapsto_{\mu_7^*}&&X_{9,12}+[2]X_{3,9,12}+X_{3^2,9,12}+X_{3^2,8,9,12}+X_{3^2,7,8,9,12}&(3)\\
&\mapsto_{\mu_{11}^*}&&X_{9,11,12}+[2]X_{3,9,11,12}+X_{3^2,9,11,12}+X_{3^2,8,9,11,12}+X_{3^2,8,9,11^2,12}+X_{3^2,7,8,9,11^2,12}&(4)\\
&\mapsto_{\mu_2^*\mu_8^*\mu_2^*}&&X_{9,11,12} + [2] X_{2,9,11,12} +  X_{2^2,9,11,12} + [2] X_{2,3,9,11,12} + [2]X_{2^2,3,9,11,12}+X_{2^2,3^2,9,11,12}&(5)\\
&&&+X_{2^2,8,9,11,12}+[2]X_{2^2,3,8,9,11,12}+X_{2^2,3^2,8,9,11,12}+X_{2^2,3^2,8,9,11^2,12}+X_{2^2,3^2,7,8,9,11^2,12}\\
&\mapsto_{\mu_6^*}&&X_{9,11,12}+[2]X_{2,9,11,12}+X_{2^2,9,11,12}+[2]X_{2,3,9,11,12}+[2]X_{2^2,3,9,11,12}+X_{2^2,3^2,9,11,12}&(6)\\
&&&+X_{2^2,8,9,11,12}+[2]X_{2^2,3,8,9,11,12}+X_{2^2,3^2,8,9,11,12}+X_{2^2,6,8,9,11,12}+[2]X_{2^2,3,6,8,9,11,12}\\
&&&+X_{2^2,3^2,6,8,9,11,12}+X_{2^2,3^2,8,9,11^2,12}+X_{2^2,3^2,6,8,9,11^2,12}+X_{2^2,3^2,6,7,8,9,11^2,12}\\
&\mapsto_{\mu_7^*}&&X_{7,9,11,12}+[2]X_{2,7,9,11,12}+X_{2^2,7,9,11,12}+[2]X_{2,3,7,9,11,12}+[2]X_{2^2,3,7,9,11,12}&(7)\\
&&&+X_{2^2,3^2,7,9,11,12}+X_{2^2,7,8,9,11,12}+[2]X_{2^2,3,7,8,9,11,12}+X_{2^2,3^2,7,8,9,11,12}\\
&&&+X_{2^2,7^2,8,9,11,12}+[2]X_{2^2,3,7^2,8,9,11,12}+X_{2^2,3^2,7^2,8,9,11,12}+X_{2^2,6,7^2,8,9,11,12}\\
&&&+[2]X_{2^2,3,6,7^2,8,9,11,12}+X_{2^2,3^2,6,7^2,8,9,11,12}+X_{2^2,3^2,7^2,8,9,11^2,12}+X_{2^2,3^2,6,7^2,8,9,11^2,12}\\
&\mapsto_{\mu_3^*\mu_{11}^*\mu_3^*}&&X_{7,9,11,12}+[2]X_{2,7,9,11,12}+X_{2^2,7,9,11,12}+[2]X_{2,3,7,9,11,12}+[2]X_{2^2,3,7,9,11,12}&(8)\\
&&&+X_{2^2,3^2,7,9,11,12}+X_{2^2,7,8,9,11,12}+[2]X_{2^2,3,7,8,9,11,12}+X_{2^2,3^2,7,8,9,11,12}\\
&&&+X_{2^2,3^2,7^2,8,9,11,12}+X_{2^2,3^2,6,7^2,8,9,11,12}\\
&\mapsto_{\mu_2^*\mu_8^*\mu_2^*}&&X_{7,9,11,12}+[2]X_{3,7,9,11,12}+X_{3^2,7,9,11,12}+X_{3^2,7,8,9,11,12}+X_{3^2,7^2,8,9,11,12}+X_{3^2,6,7^2,8,9,11,12}&(9)\\
&\mapsto_{\mu_7^*}&&X_{9,11,12}+[2]X_{3,9,11,12}+X_{3^2,9,11,12}+X_{3^2,8,9,11,12}+X_{3^2,6,8,9,11,12}&(10)\\
&\mapsto_{\mu_6^*}&&X_{9,11,12}+[2]X_{3,9,11,12}+X_{3^2,9,11,12}+X_{3^2,8,9,11,12}&(11)\\
&\mapsto_{\mu_3^*\mu_8^*\mu_3^*}&&X_{9,11,12}&(12)\\
&\mapsto_{\mu_{11}^*}&&X_{9,12}&(13)\\
&\mapsto_{\mu_7^*}&&X_{9,12}
}
}

{\tiny
\Eqn{
X_{4,9,12}&\mapsto_{\mu_2^*\mu_6^*\mu_2^*}&&X_{4,9,12}&(1)\\
&\mapsto_{\mu_3^*\mu_8^*\mu_3^*}&&X_{3,4,9,12}+X_{3^2,4,9,12}+X_{3^2,4,8,9,12}&(2)\\
&\mapsto_{\mu_7^*}&&X_{3,4,9,12}+X_{3^2,4,9,12}+X_{3^2,4,8,9,12}+X_{3^2,4,7,8,9,12}&(3)\\
&\mapsto_{\mu_{11}^*}&&X_{3,4,9,11,12}+X_{3^2,4,9,11,12}+X_{3^2,4,8,9,11,12}+X_{3^2,4,8,9,11^2,12}+X_{3^2,4,7,8,9,11^2,12}&(4)\\
&\mapsto_{\mu_2^*\mu_8^*\mu_2^*}&&X_{2,3,4,9,11,12}+X_{2^2,3,4,9,11,12}+X_{2^2,3^2,4,9,11,12}+X_{2^2,3,4,8,9,11,12}+X_{2^2,3^2,4,8,9,11,12}&(5)\\
&&&+X_{2^2,3^2,4,8,9,11^2,12}+X_{2^2,3^2,4,7,8,9,11^2,12}\\
&\mapsto_{\mu_6^*}&&X_{2,3,4,9,11,12}+X_{2^2,3,4,9,11,12}+X_{2^2,3^2,4,9,11,12}+X_{2^2,3,4,8,9,11,12}+X_{2^2,3^2,4,8,9,11,12}&(6)\\
&&&+X_{2^2,3,4,6,8,9,11,12}+X_{2^2,3^2,4,6,8,9,11,12}+X_{2^2,3^2,4,8,9,11^2,12}+X_{2^2,3^2,4,6,8,9,11^2,12}\\
&&&+X_{2^2,3^2,4,6,7,8,9,11^2,12}\\
&\mapsto_{\mu_7^*}&&X_{2,3,4,7,9,11,12}+X_{2^2,3,4,7,9,11,12}+X_{2^2,3^2,4,7,9,11,12}+X_{2^2,3,4,7,8,9,11,12}+X_{2^2,3^2,4,7,8,9,11,12}&(7)\\
&&&+X_{2^2,3,4,7^2,8,9,11,12}+X_{2^2,3^2,4,7^2,8,9,11,12}+X_{2^2,3,4,6,7^2,8,9,11,12}+X_{2^2,3^2,4,6,7^2,8,9,11,12}\\
&&&+X_{2^2,3^2,4,7^2,8,9,11^2,12}+X_{2^2,3^2,4,6,7^2,8,9,11^2,12}\\
&\mapsto_{\mu_3^*\mu_{11}^*\mu_3^*}&&X_{2,3,4,7,9,11,12}+X_{2^2,3,4,7,9,11,12}+X_{2^2,3^2,4,7,9,11,12}+X_{2^2,3,4,7,8,9,11,12}+X_{2^2,3^2,4,7,8,9,11,12}&(8)\\
&&&+X_{2^2,3^2,4,7^2,8,9,11,12}+X_{2^2,3^2,4,6,7^2,8,9,11,12}\\
&\mapsto_{\mu_2^*\mu_8^*\mu_2^*}&&X_{3,4,7,9,11,12}+X_{3^2,4,7,9,11,12}+X_{3^2,4,7,8,9,11,12}+X_{3^2,4,7^2,8,9,11,12}+X_{3^2,4,6,7^2,8,9,11,12}&(9)\\
&\mapsto_{\mu_7^*}&&X_{3,4,9,11,12}+X_{3^2,4,9,11,12}+X_{3^2,4,8,9,11,12}+X_{3^2,4,6,8,9,11,12}&(10)\\
&\mapsto_{\mu_6^*}&&X_{3,4,9,11,12}+X_{3^2,4,9,11,12}+X_{3^2,4,8,9,11,12}&(11)\\
&\mapsto_{\mu_3^*\mu_8^*\mu_3^*}&&X_{4,9,11,12}&(12)\\
&\mapsto_{\mu_{11}^*}&&X_{4,9,12}&(13)\\
&\mapsto_{\mu_7^*}&&X_{4,9,12}
}
}

This completes the discussion in type $B_3$.

\end{document}